\documentclass[reqno,11pt]{amsart}
\usepackage{amsmath, latexsym, amsfonts, amssymb, amsthm, amscd}
\usepackage[utf8]{inputenc}
\usepackage{graphics,epsf,psfrag,epsfig}

\numberwithin{equation}{section}

\newtheorem{theorem}{Theorem}[section]
\newtheorem{lemma}[theorem]{Lemma}
\newtheorem{proposition}[theorem]{Proposition}
\newtheorem{cor}[theorem]{Corollary}
\newtheorem{rem}[theorem]{Remark}

\renewcommand{\ge}{\geq}
\renewcommand{\le}{\leq}
\newcommand{\ind}{\mathbf{1}}
\newcommand{\E}{\mathbb{E}}
\newcommand{\R}{\mathbb{R}}

\newcommand{\N}{\mathbb{N}}
\renewcommand{\tilde}{\widetilde}
\renewcommand{\hat}{\widehat}

\newcommand{\cT}{{\ensuremath{\mathcal T}} }

\newcommand{\bP}{{\ensuremath{\mathbf P}} }
\newcommand{\bE}{{\ensuremath{\mathbf E}} }

\DeclareMathSymbol{\leqslant}{\mathalpha}{AMSa}{"36} 
\DeclareMathSymbol{\geqslant}{\mathalpha}{AMSa}{"3E} 
\DeclareMathSymbol{\eset}{\mathalpha}{AMSb}{"3F}     
\renewcommand{\leq}{\;\leqslant\;}                   
\renewcommand{\geq}{\;\geqslant\;}                   
\newcommand{\dd}{\,\text{\rm d}}             

\newcommand{\bbE}{{\ensuremath{\mathbb E}} }

\newcommand{\bbL}{{\ensuremath{\mathbb L}} }

\newcommand{\bbN}{{\ensuremath{\mathbb N}} }

\newcommand{\bbP}{{\ensuremath{\mathbb P}} }

\newcommand{\bbR}{{\ensuremath{\mathbb R}} }

\newcommand{\ga}{\alpha}
\newcommand{\gb}{\beta}
\newcommand{\gd}{\delta}
\newcommand{\gep}{\varepsilon}       

\newcommand{\gz}{\zeta}
\newcommand{\gG}{\Gamma}

\newcommand{\go}{\omega}

\newcommand{\gl}{\lambda}

\newcommand{\gs}{\sigma}

\makeatletter
\def\captionfont@{\footnotesize}
\def\captionheadfont@{\scshape}

\long\def\@makecaption#1#2{%
  \vspace{2mm}
  \setbox\@tempboxa\vbox{\color@setgroup
    \advance\hsize-6pc\noindent
    \captionfont@\captionheadfont@#1\@xp\@ifnotempty\@xp
        {\@cdr#2\@nil}{.\captionfont@\upshape\enspace#2}%
    \unskip\kern-6pc\par
    \global\setbox\@ne\lastbox\color@endgroup}%
  \ifhbox\@ne 
    \setbox\@ne\hbox{\unhbox\@ne\unskip\unskip\unpenalty\unkern}%
  \fi
  \ifdim\wd\@tempboxa=\z@ 
    \setbox\@ne\hbox to\columnwidth{\hss\kern-6pc\box\@ne\hss}%
  \else 
    \setbox\@ne\vbox{\unvbox\@tempboxa\parskip\z@skip
        \noindent\unhbox\@ne\advance\hsize-6pc\par}%
\fi
  \ifnum\@tempcnta<64 
    \addvspace\abovecaptionskip
    \moveright 3pc\box\@ne
  \else 
    \moveright 3pc\box\@ne
    \nobreak
    \vskip\belowcaptionskip
  \fi
\relax
}
\makeatother
\def\writefig#1 #2 #3 {\rlap{\kern #1 truecm
\raise #2 truecm \hbox{#3}}}

\newcommand{\tf}{\textsc{f}}

\newcommand{\ttau}{\tilde{\tau}}

\newcommand{\tT}{\tilde{T}}

\newcommand{\tga}{\tilde{\alpha}}

\renewcommand{\a}{\mathrm{ann}}
\newcommand{\p}{\mathrm{pin}}
\newcommand{\q}{\mathrm{que}}
\newcommand{\pin}{\mathrm{pin}}
\newcommand{\ann}{\mathrm{ann}}
\newcommand{\htau}{\hat\tau}

\title[pinning model in random correlated environment]{Sharp critical behavior for pinning model in random correlated environment}

\author{Quentin Berger}
\address{
Laboratoire de Physique, ENS Lyon,  Universit\'e de Lyon, 46 All\'ee d'Italie, 
69364 Lyon, France
}
\email{quentin.berger@ens-lyon.fr}

\author{Hubert Lacoin}
\address{CEREMADE, Université Paris Dauphine, Place du Maréchal De Lattre De Tassigny, 75775 Paris Cedex 16 - France}
\email{lacoin@ceremade.dauphine.fr}

\begin{document}

\begin{abstract}
  This article investigates the effect for random pinning models of
  long range power-law decaying correlations in the environment. For a
  particular type of environment based on a renewal construction, we
  are able to sharply describe the phase transition from the
  delocalized phase to the localized one, giving the critical exponent
  for the (quenched) free-energy, and proving that at the critical
  point the trajectories are fully delocalized.  These results
  contrast  with what happens both for  the pure model (i.e. without
  disorder) \cite{Fisher} and for the widely studied case of i.i.d.\ disorder,
  where the relevance or irrelevance of disorder on the critical properties
  is decided via the so-called Harris Criterion \cite{A06,
    AZ08, DGLT07, GLT08,GLT09}.
\end{abstract}

\maketitle

\section{Introduction}

\subsection{Physical motivations}
The effect of disorder long-range correlations on the critical
properties of a physical system has been well-studied in the physics
literature, historically in \cite{WeinHalp83} for a general class of
models, and in \cite{Usatenko} for the phenomenon we are interested
in: the adsorption of a polymer on a wall or a line.  One example that
arises in nature is the DNA sequence, that has been found
\cite{Li,Peng} to exhibit long-range power-law correlations and it is
thought that some repetitive patterns are responsible for these
correlations. It is
of great interest to analyse how these correlations affect the DNA denaturation process.
 We study here a probabilistic model that represents a
polymer which is pinned on a line that presents strongly correlated
disorder with repetitive (but not periodic) patterns, and we show that,
according to physicists' predictions, the critical properties of the
model are modified with respect to the case where the disorder is independent
at each site of the line.

\subsection{Definition of the model}

%
%

Let $\tau:=\{\tau_n\}_{n\geq0}$ be a recurrent renewal sequence, that
is a sequence of random variables such that $\tau_0:=0$, and
$\{\tau_{i+1}-\tau_{i}\}_{i\geq 0}$ are independent random variables
identically distributed with support in $\N$, with common law (called
inter-arrival distribution) denoted by $K(\cdot)$. The law of $\tau$
is denoted by $\bP$.  We assume that $K(\cdot)$ satisfies
\begin{equation}
 K(n):=\bP(\tau_1=n) = (1+o(1))\frac{c_K}{n^{1+\ga}},
\label{assump:K}
\end{equation} 
for some $\alpha>0$, $\alpha\ne 1$ (the assumption 
$\alpha\ne1$ does not
hide anything  deep but it avoids various technical nuisances).
The fact that the renewal is recurrent simply means that
$K(\infty)=\bP(\tau_1=+\infty)=0$.
We assume also for simplicity that $K(n)>0$ for all $n\in \N$.
We use the notation
\begin{equation}
\bar K(n):=\bP(\tau_1>n)=\sum_{i=n+1}^{\infty} K(i).
\end{equation}
With a slight abuse of notation, $\tau$ also denotes the set $\{ k\in \N \ |\ \tau_n=k \text{ for some } n\}$.
Given a sequence $\go=(\go_n)_{n\in \N}$ of real numbers (the environment), $h\in \bbR$ (the pinning parameter)
and $\gb\ge 0$ (the inverse temperature),
we define the sequence of polymer measures $\bP_{N,h}^{\go,\gb}$, $N\in \N$ as follows
\begin{equation}
 \frac{\dd \bP_{N,h}^{\go,\gb}}{\dd \bP} (\tau) := 
\frac{1}{Z_{N,h}^{\go,\gb}}   \exp\left( \sum_{n=1}^N  (h+\gb\go_n) \ind_{\{n\in\tau\}}\right),
\end{equation}
where 
\begin{equation}
 Z_{N,h}^{\go,\gb}:= \bE\left[ \exp\left( \sum_{n=1}^N  (h+\gb\go_n) \ind_{\{n\in\tau\}}\right) \right]
\end{equation}
is called the \emph{partition function} of the system.

The set $\tau$ can be thought of as the set of return times to its
departure point (call it $0$)  of some random walk $S$ on some state space, say
$\mathbb Z^d$.  The graph of the random walk $(k,S_k)_{k\in[0,N]}$ is
interpreted as a $1$-dimensional polymer chain living in a
$(d+1)$-dimensional space, and interacting with the defect line
$[0,N]\times\{0\}$.  Physically, our modification of $\bP$ corresponds
to giving an energy reward (or penalty, depending on its sign) to the
trajectory $(k,S_k)_{k\in[0,N]}$ when it touches the defect line, at
the times $(\tau_i)_{i\in\N}$. The reward consists of an homogeneous
part: $h$, and an inhomogeneous one: $\gb\go_n$.

\medskip

Our aim is to study the properties of $\tau\cap[0,N]$ under the
polymer measure $\bP_{N,h}^{\go,\gb}$ for large values of $N$.  This
model, known as \textit{inhomogeneous pinning model}, has been studied
in depth in the literature (see \cite{Book, SFLN} for complete
reviews on the subject), in particular in the cases where $\go$ is a
periodic sequence \cite{BG,CGZ1,CGZ2} and where $\go$ is a typical
realization of a sequence of i.i.d.\ variables
\cite{A06,AZ08,DGLT07,GT05,GLT08,GLT09}.  In this paper, we focus on a
particular type of environment $\go$, constructed as follows:

Let $\hat \tau=(\hat \tau_n)_{n\ge 0}$, $\hat \tau_0=0$ be a recurrent
renewal process (let $\hat \bP$ denote its law), with inter-arrival
law $\hat K(\cdot)$ that satisfies
\begin{equation}
  \hat K(n):=\hat \bP(\hat\tau_1=n) = (1+o(1))\frac{\hat c_K}{n^{1+\tilde \ga}},
\end{equation}
for some $\tilde \ga>1$.
These conditions ensure that $\hat \bE[ \hat\tau_1]<\infty$ which is crucial.
Then let 
$(X_i)_{i\ge 1}$ be a sequence of i.i.d.\ random variables (law $\bbP$ independent of $\hat \bP$) satisfying
\begin{equation}
 \bbP(X_i=0)=\bbP(X_i=-1)=1/2
\end{equation}
and set
\begin{equation}
\go_n= X_i, \quad \forall n \in (\tau_{i-1},\tau_{i}].
\label{defenvir}
\end{equation}

For later convenience we may use another construction to get $\go$.
We start from the renewal process $\tilde \tau$ (let $\tilde \bP$ denote its law), with inter-arrival law $\tilde K(\cdot)$ given by
\begin{equation}\label{deftiled}
\tilde \bP (\tilde\tau_1=n):=\tilde K(n):=\sum_{k=1}^\infty 2^{-k}\hat \bP( \hat \tau_k=n).
\end{equation}
One can check (using Proposition \ref{doneybis} in the appendix), that 
\begin{equation}\label{tidd}
  \tilde K(n)= (1+o(1))\frac{2 \hat c_K}{n^{1+\tilde \ga}}.
\end{equation}
Then one sets 
\begin{equation}\label{altern}
    \go_i= \begin{cases}\, \, 0 & \text{ if there exists some } n\geq 0 \text{ such that } i\in( \ttau_{2n}, \ttau_{2n+1}], \\
    -1 & \text{ if there exists some } n\geq 0 \text{ such that } i \in (\ttau_{2n+1}, \ttau_{2n+2}].
 \end{cases}
\end{equation} 
This construction gives an environment with the same law as the first one conditioned to $X_1=0$,
and this conditioning is harmless for our purpose.

\medskip
\begin{rem}\rm
The reason to choose such an environment is that it is a simple framework to study the influence
 of long-range power-law correlations for disordered pinning models.
One can compute the correlation easily: for any $i\in\N$, $k\geq 0$
\begin{equation}
 \mathrm{Cov}(\go_i,\go_{i+k}) = \frac{1}{4} \hat\bP \left( \exists n \in\N ,\  (i;i+k) \in( \hat\tau_{n-1},\hat\tau_n ]^2\right).
\end{equation}
The latter term is equal to
\begin{equation}
 \frac14 \sum_{l=1}^{i-1} \hat\bP(l\in\hat\tau)  \hat \bP(\hat\tau_1 > k+i-l) \stackrel{k\to\infty}{\sim}
       \frac{\hat c_K}{4\tilde\ga} \sum_{l=1}^{i-1} \hat\bP(l\in\hat\tau) (i-l+k)^{-\tilde \alpha}.
\end{equation}
One uses the renewal theorem to get that
$\hat\bP(l\in\hat\tau)\sim_{l\to\infty} \hat\bE[\hat \tau_1]^{-1}$,
so that taking $i$ large, one has that $\mathrm{Cov}(\go_i,\go_{i+k})$ is of order $k^{1-\tilde\ga }$, 
which decays slower and slower as $\tilde\ga$ is taken
close to $1$.

The reason why we impose $\tilde\alpha>1$ is that for $\tilde\alpha<1$ the
model is somewhat trivial. Indeed, in that case, the infinite-volume quenched (averaged) free energy
has the same critical behavior as for the non-disordered model. Moreover,
in this case one loses the ergodicity of the environment sequence and the free energy is no more a self-averaging quantity (i.e. the 
almost sure limit in \eqref{eq:qwert} does not exist).

\end{rem}

\subsection{The homogeneous model}

Before giving our results, we recall some facts about the easier case $\gb=0$, that is called \textit{homogeneous pinning model}.
This model presents the particularity of being exactly solvable (see \cite{Fisher}).
Recall the definition of the polymer measure in this particular case:
\begin{equation}
 \frac{\dd \bP_{N,h}}{\dd \bP} (\tau) := \frac{1}{Z_{N,h}}   \exp\left( \sum_{n=1}^N h \ind_{\{n\in\tau\}}\right),
\end{equation}
where
\begin{equation}
 Z_{N,h}:= \bE\left[ \exp\left( \sum_{n=1}^N  h \ind_{\{n\in\tau\}}\right) \right]
\end{equation}
is the \textsl{partition function}, \textit{i.e.} the normalizing factor that makes $\bP_{N,h}$
a probability measure. The study of the asymptotics of the partition function allows to describe the typical
behavior of $\tau\cap [0,N]$ under $\bP_{N,h}$ for large $N$. We summarize this fact in the following Proposition.

\begin{proposition}[\cite{Book}, Chapter 2]
 The limit 
\begin{equation}
 \tf(h):= \lim_{N \to \infty} \frac{1}{N}\log Z_{N,h}
\end{equation}
exists and is called \textsl{free energy}. Moreover $h\mapsto \tf(h)$ is a non-decreasing convex function, and
$\tf(h)>0$ if and only if $h>0$. One also has the following asymptotics of $\tf(h)$ around $h=0_+$:
\begin{equation}
\label{th:puro}
 \tf(h)= \begin{cases} \frac{\alpha h^{1/\ga}}{\gG(1-\alpha)c_k}(1+o(1)) \text { if } \alpha<1,\\
                       (\bE[\tau_1])^{-1} h(1+o(1)) \text{ if } \alpha>1. 
         \end{cases}
\end{equation}
Moreover, at every point where $\tf$ is differentiable one has 
\begin{equation}
 \lim_{N\to \infty}\frac{1}{N}\bE_{N,h}\left[\sum_{n=1}^N \ind_{\{n\in \tau\}}\right]=\tf'(h).
\end{equation}
\label{homo1}
\end{proposition}
The above result implies that the number of contact points $|\tau \cap [0,N]|$ under the polymer measure is 
of order $N$ for $h>0$ (and also for $h=0$, $\alpha>1$ by the renewal Theorem \cite[Chapter 1, Theorem 2.2]{Asmuss})
 and $o(N)$ in the other cases.
In fact one can get a more precise statement.

\begin{proposition}[\cite{Book}]\label{homo2} (Asymptotic behavior of the path measure)

\begin{itemize}
\item When $h<0$, for all $k$ one has
\begin{equation}
 \lim_{N\to \infty}\bP_{N,h}(|\tau \cap[0,N]|=k+1)= (1-e^{h}) e^{kh}.
\end{equation}

\item When $h=0$, and $\alpha\in(0,1)$ one has that under $\bP=\bP_{N,h=0}$
\begin{equation}
 N^{-\alpha} |\tau\cap [0,N]| \Rightarrow \mathcal A_{\alpha},
\end{equation}
where $A_{\alpha}$ is the inverse of an $\alpha$-stable law.
\end{itemize}
\end{proposition}

\subsection{Preliminary results on the disordered model}
This paper presents results for our inhomogeneous model that exhibits sharp
contrast with Proposition \ref{homo1} and \ref{homo2}. We show that
disorder modifies the phase transition between the localized phase
(order $N$ contacts, positive free energy), and the delocalized phase,
($O(1)$ contacts, zero free energy). Due to the correlations present in the
environment, this phenomenon is very different from what was observed
for the i.i.d.\ environment case.


In order to state our results, we first need to show the existence of
the free energy for the inhomogeneous model.

\begin{proposition}
\label{existF}
The limit
\begin{equation}
\label{eq:qwert}
\tf(\gb,h):= \lim_{N\to\infty} \frac{1}{N} \log Z_{N,h}^{\go,\gb},
\end{equation}
exists $\bbP\times \hat \bP$ almost surely. One has
\begin{equation}
\tf(\gb,h)=\lim_{N\to\infty} \frac{1}{N} \hat \bE \ \bbE \left[\log Z_{N,h}^{\go,\gb}\right].
\end{equation}
The function $h\to \tf(\gb,h)$ is non-decreasing, non-negative and convex.
At every point where $\tf$ has a derivative one has $\bbP\times \hat \bP$ a.s.
\begin{equation}
 \lim_{N\to \infty}\frac{1}{N}\bE^{\omega,\beta}_{N,h}\left[\sum_{n=1}^N \ind_{\{n\in \tau\}}\right]=\frac{\partial}{\partial h}\tf(\gb,h).
\end{equation}
\end{proposition}
\begin{proof}
 The second part of the result is classic for pinning models (see for example \cite{Book}) and we leave it to the reader.
 For the first one, one introduces the partition function with \textsl{pinned} boundary condition:

\begin{equation}\label{pbc}
 Z_{N,h}^{\go,\gb,\pin}:= \bE\left[ \exp\left( \sum_{n=1}^N  (h+\gb\go_n) \ind_{\{n\in\tau\}}\right)\ind_{\{N\in \tau\}} \right].
\end{equation}
Note that (see equation (4.25) in \cite{Book} and its proof) there exists a constant $c>0$ such that 
\begin{equation}\label{houhi}
  c N^{-1} e^{-\gb+h}Z_{N,h}^{\go,\gb}  \le Z_{N,h}^{\go,\gb,\pin}\le  Z_{N,h}^{\go,\gb},
\end{equation}
so that it is equivalent to work with $Z$ or $Z^{\pin}$ as far as $\tf$ is concerned.
Then one notices that
\begin{equation}
 Z_{N+M,h}^{\go,\gb,\pin}
\ge \bE\left[ \exp\left( \sum_{n=1}^N  (h+\gb\go_n) \ind_{\{n\in\tau\}}\right)\ind_{\{N \in \tau,\ (N+M)\in \tau\}} \right]
=  Z_{N,h}^{\go,\gb,\pin}Z_{M,h}^{\theta^N \go,\gb,\pin},
\end{equation}
where $\theta$ is the shift operator, \textit{i.e.} $\theta^N \go:=(\go_{n+N})_{n\ge 0}$. So that in particular
\begin{equation} \label{aftshd}
 \log Z_{\hat\tau_{N+M},h}^{\go,\gb,\pin}\ge \log Z_{\hat \tau_N,h}^{\go,\gb,\pin}
+\log Z_{\hat \tau_{M+N}-\hat\tau_{N},h}^{\theta^{\hat \tau_N} \go,\gb,\pin}.
\end{equation}
Note that from the renewal construction of the environment,
the two terms on the right hand-side are independent and that the law of the second one is the same 
as the law of
 $\log Z_{\hat \tau_M,h}^{\go,\gb,\pin}$.
Therefore one can use Kingman's superadditive ergodic Theorem \cite[Theorem 1]{King} or 
simply the law of large numbers (like it is done in \cite[Section 4.2]{Book}) to conclude that 
\begin{multline}\label{altfreen}
 \lim_{N\to \infty} \frac{1}{N} \log Z_{\hat \tau_N,h}^{\go,\gb,\pin}
= \lim_{N\to \infty} \frac{1}{N}\hat\bE \ \bbE \left[\log Z_{\hat \tau_N,h}^{\go,\gb,\pin}\right]\\
=\sup_{N\ge 0} \frac{1}{N}\hat\bE\ \bbE  \left[\log Z_{\hat \tau_N,h}^{\go,\gb,\pin}\right]=:\bar\tf(\gb,h).
\end{multline}
Then the law of large numbers for $\hat \tau$ gives that
\begin{equation}\label{altfreen2}
 \tf(\gb,h)=\lim_{N\to \infty} \frac{1}{\hat\tau_N} \log Z_{\hat \tau_N,h}^{\go,\gb,\pin}=
\lim_{N\to\infty} \frac{N}{\hat\tau_N}\lim_{N\to\infty}  \frac{1}{N}\bbE \hat\bE \left[\log Z_{\hat \tau_N,h}^{\go,\gb,\pin}\right]=
\frac{1}{\hat \bE\left[\hat \tau_1\right]}\bar \tf(\gb,h).
\end{equation}
Note that we have proved only convergence almost surely along the
random subsequence $\hat \tau$. Then one can use standard arguments to
show that convergence holds for the whole sequence and also in
$\bbL_1$ (details are omitted).

\end{proof}

A matter of interest for disordered pinning models in the i.i.d.\ environment case is how the free-energy compares 
with the annealed free-energy defined by
\begin{equation}
 \tf^{\ann}(\gb,h):=\lim_{N\to \infty}\frac{1}{N}\log \hat \bE \ \bbE\left[Z_{N,h}^{\go,\gb}\right].
\end{equation}
Jensen's inequality gives that $\tf(\gb,h)\le \tf^{\ann}(\gb,h)$.
In our case this bound does not give much information. Indeed,
\begin{equation}
Z_{N,h}\ge \hat \bE \ \bbE\left[ Z_{N,h}^{\go,\gb}\right]\ge \frac12 \hat \bP(\hat \tau_1> N)   Z_{N,h}.
\end{equation}
As $\hat\bP[\hat \tau_1> N]$ behaves like $N^{-\tilde \alpha}$ for $N$ large, this
factor does not affect
the limit after taking the $\log$ and dividing by $N$. Therefore, $\tf^{\ann}(\gb,h)=\tf(h)$ and the annealed bound 
for the free-energy becomes simply
\begin{equation}
 \tf(\gb,h)\le \tf(h),
\end{equation}
which is obvious from monotonicity in $\go$ of $Z_{N,h}^{\go,\gb}$.
This contrasts with the case of i.i.d.\ environment, for which the annealed bound gives a non-trivial upper-bound on the free-energy.

\subsection{Main results and a comparison with the  previous literature}

What we show concerning the free-energy of  our disordered model is
that it is positive for every positive $h$ (\textit{i.e.} that the
presence of negative $\go$ is not sufficient to repel the trajectories
from the defect line).  Moreover, we are able to compute the
asymptotics of the free-energy around $h=0_+$ up to a constant.

\begin{theorem}\label{freeE}
There exist two constants $C_1>0$ and $C_2>0$ (depending on $\gb$), such that for any $h\in (0,1)$, one has
\begin{equation}\label{fre1}
  C_1 h^{\frac{\tilde \alpha}{(1\wedge \alpha)}} |\log h|^{1-\tga} \leq \tf(\gb,h)
\leq C_2 h^{\frac{\tilde \alpha}{(1\wedge \alpha)}} |\log h|^{1-\tga}.
\end{equation}
\end{theorem}

\begin{rem}\rm
 Note that in the statement of the theorem the constants depend on $\gb$. This will be the case of many constants introduced during the proof, 
and we may not mention it, as in the sequel we always consider $\gb$ as a fixed parameter.
\end{rem}

Our second result is that at the critical point $h=0$, the
trajectories are strictly delocalized in the sense that typical
trajectories have only finitely many returns to zero.

\begin{theorem}\label{ncontacts}
The sequence of law $(\nu_N)_{N\ge 0}$ on $\bbN$ defined by  
\begin{equation}
\nu_N(A):=\bP_{N,h=0}^{\gb,\go}(|\tau\cap[0,N]|\in A),
\end{equation} 
(the laws of the number of contact under $\bP_{N,h=0}^{\gb,\go}$
is tight for almost every realization of $\go$. 
\end{theorem}

We prove this result in Section \ref{sec:ncontactvrai}, and actually we get a more precise result
in Corollary~\ref{cor:ncontactvrai} and Proposition~\ref{lowertail} that we sum up as follows.
\begin{proposition}
\label{prop:tail}
For almost every $\go$, for any $\gep>0$ there exists $a_0=a_0(\go,\gb,\gep)\in\R$ such that for $a\geq a_0$ and $a\leq N^{\frac{(1/\ga)\wedge \ga}{\tilde \ga}-\gep}$
one has
\begin{equation}
 a^{-\gep -\frac{\tilde\ga(\ga+1)-1}{1\wedge \ga}} \leq \bP_{N,h=0}^{\go,\gb} \left( |\tau\cap[0,N]|=a \right) \leq a^{\gep - \tilde{\ga} (1\vee \ga)}.
\end{equation} 
\end{proposition}


\begin{rem}\rm
\label{rem:tail}
Proposition \ref{prop:tail} indicates that the asymptotic law of the number
of contacts under $\bP_{N,h=0}^{\go,\gb}$ has a power-decaying tail. 
This power-law behavior contrasts with what happens for $h<0$, where the law of ${|\tau\cap [0,N]|}$ has an exponential tail.
In view of how our results are obtained, we conjecture that it is the lower-bound given in Proposition \ref{prop:tail} that is sharp.
\end{rem}

\medskip

It is instructive to compare the sharp estimates of Theorems
\ref{freeE} and \ref{ncontacts} with the results available in the
literature on other pinning models.

The first important remark is that the free energy critical exponent
(call it $\nu$, so that $\nu= \tilde \alpha/(1\wedge \alpha)$,
cf. \eqref{fre1}) is different both from the critical exponent of the
homogeneous model: $\nu=1/(1\wedge \alpha)$ (cf. Proposition
\ref{homo1}) and from that of the disordered model with
i.i.d.\ disorder. In the latter the critical exponent equals
$\nu=1/\alpha$ if $\alpha<1/2$ and $\beta$ small (regime of irrelevant
disorder \cite{A06,T}) and in all cases (every $\alpha,\beta>0$) one
observes a disorder induced \textit{smoothing} of the free-energy curve near the critical point that implies $\nu\ge2$ when it exists \cite{GT05}
(in contrast, remark that
the critical exponent in \eqref{fre1} can be smaller than $2$
for our correlated model).
Always concerning the critical exponent, let us also add that up to
now precise asymptotics of the free-energy (close to
the critical point) for pinning models  had been proved only for the case of homogeneous
(or weakly inhomogeneous, i.e.\ periodic) environment (Proposition
\ref{homo1}), and for the mentioned case of i.i.d.\ environment,
$\alpha<1/2$ and $\beta$ small \cite{A06, T, GT_ap} (we let aside
\cite{Aivy} where it is proved that first order transition occurs for
a very special model).

A second important observation concerns the value of the critical point.
In our model, it equals zero for the homogeneous model (and therefore for the annealed one) but also for the quenched model 
(for every $\alpha,\tilde\alpha,\beta$). 
This is in contrast with what happens for i.i.d.\ random environment:
in that case, the critical point of the annealed model
equals  $h_c^{\a}(\beta) = -\log \E[e^{\gb\go_1}]$. Also, for i.i.d.\ environment it is a 
crucial issue to know whether the critical point $h_c^{\q}(\beta)$ of the
quenched model coincides or not with $h_c^\a(\beta)$: one has
 $h_c^{\q}(\beta)=h_c^{\a}(\beta)$ if $\ga<1/2$, $\beta$ small \cite{A06,T}
and $h_c^{\q}(\beta)<h_c^{\a}(\beta)$ if $\ga\geq 1/2$ (every $\beta>0$, 
with sharp bounds on their difference in the limit of $\beta$ small \cite{AZ08,DGLT07,GLT08,GLT09});
another situation where  $h_c^{\q}(\beta)<h_c^{\a}(\beta)$ 
is  $\ga<1/2$, $\beta$ large \cite{T_AAP}.



Finally, we make some observations concerning the behavior of the trajectories
at the critical point given by Theorem \ref{ncontacts}. The exact behavior is
known for the pure model (cf.\ Proposition \ref{homo2}), in the irrelevant
disorder regime for i.i.d.\ disorder (see \cite{L}), but very little
is known in the other cases (in \cite{L} it is shown that there should
be at most $N^{1/2+\gep}$ contacts with large probability, this result
being linked to the above mentioned free energy critical exponent
bound $\nu\ge2$).  In contrast, in our model the number of contacts at
the critical point is not directly related to the critical behavior of
the free energy (see however Proposition \ref{prop:tail}).
Note that up to now, for i.i.d.\ disordered pinning
models, the best general bound one has for the number of contact points in the
delocalized phase is $O(\log N)$ \cite[Section 8.2]{Book}, but in our
case one has that it is $O(1)$.

Concerning previous results on pinning models with correlated random
environment, the only work we are aware of is \cite{Poisat}, where a
model with \emph{finite-range} disorder
correlations is studied.
Let us also mention that the authors of
\cite{BL,BergToni,BS08,BS09} consider a random walk that is pinned on
a second (quenched) random walk: this can also be seen as an example
of a pinning model in a correlated environment. In both of this cases, however,
the results one finds are similar to the ones of the
i.i.d. environment case.

\smallskip

We have chosen to constrain ourselves only to a very particular setup
for the sake on simplicity, however our results should hold with much greater generality
for correlated environment $\go\in\{-1,0\}^{\N}$.

\subsection{Strategy of the polymer under $\bP_{N,h}^{\go,\gb}$, ideas of the proofs}

\label{sec:strategy}

  We give in this section an idea on the strategy the polymer adopts
  under the measure $\bP_{N,h}^{\go,\gb}$, this undersanding
  clarifying the schemes of the proofs of
  Theorems \ref{freeE} and \ref{ncontacts}.

\smallskip

  The proof of Theorem \ref{freeE} gives the right bounds on the free
  energy, but also a heuristic understanding of the typical behavior of the trajectories
  under the measure $\bP^{\omega,\gb}_{N,h}$.
  The idea is that the polymer tends to pin on the regions where
  $\go\equiv0$, but only those of length larger than $h^{-\frac{1}{1\wedge
      \alpha}}|\log h|$, whereas they are repelled from the interface
  by any other region.
  Thus the idea to prove Theorem \ref{freeE} is to estimate
  the contribution of all these different kinds of
  regions to the partition function. For the lower bound
  the strategy of targeting only regions of length larger
  than $h^{-\frac{1}{1\wedge\alpha}}|\log h|$ already gives the right result.
  To get the upper bound, one has to control the contribution of all the possible
  trajectories.
  Roughly, the argument is that one uses a coarse-graining argument
  to cut the system into blocks of finite size, and sees that if one block does not contain
  a region of length larger than $h^{-\frac{1}{1\wedge\alpha}}|\log h|$
  it does not contribute to the partition function.

  A consequence of this observation is that the behavior of the
  free-energy near the critical point depends on the frequency of
  occurrence of regions of length $h^{-\frac{1}{1\wedge \alpha}}|\log
  h|$ where $\go \equiv 0$. When $\tilde\alpha$ is close to one, these
  regions occur relatively frequently, and for this reason the
  critical exponent for the free-energy in our model is close to the
  one of the homogeneous model. The two exponents get more and more
  different when $\tilde \alpha$ grows and this type of regions
  becomes more rare.

\smallskip

 Now, let us explain how we intend to prove Theorem \ref{ncontacts} and
 Proposition~\ref{prop:tail}. 
 We bound from above the probability of having exactly $a$ contacts before
 $N$ under the measure $\bP_N^{\omega,\beta}$ by considering the
 contribution of the different strategies for the polymer trajectory.
 For a trajectory $\tau$, let $V^{\hat \tau}_N(\tau)$ be the number of $\hat \tau$-renewal
 stretches (we call $\hat\tau$-stretch a segment 
 of the type $(\hat \tau_i, \hat \tau_{i+1}]$) visited by $\tau$:
\begin{equation}
\label{defVtau}
 V^{\hat \tau}_N(\tau):=|\{ i\in \N \ | \ \exists \ j\in \tau\cap[0,N], j\in(\hat \tau_{i}, \hat \tau_{i+1}]\}|.
\end{equation}
 We split the set of trajectories such that $\{|\tau\cap[0,N]|=a\}$ into two cases
\begin{itemize}
\item The trajectory $\tau$ visits a lot of $\hat \tau$-stretches (say $V^{\hat \tau}_N(\tau)\ge a^{\gep}$),
\item The trajectory $\tau$ visits only a few $\hat \tau$-stretches ($V^{\hat \tau}_N(\tau)< a^{\gep}$).
\end{itemize}

 One remarks that for any trajectory $\tau$
\begin{equation}\label{pioutrrr}
 \bbE\left[e^{\sum_{n=1}^N \gb\go_n\ind_{\{n\in \tau\}}}\right]\le \left(\frac{1+e^{-\gb}}{2}\right)^{V^{\hat \tau}_N(\tau)},
\end{equation}
 where we recall that $\bbE$ denotes the average only on the values of 
 $\{X_i\}_{i\in\N}$, i.e. on the disorder $\omega$ \emph{conditionally on the
 realization of $\hat \tau$.}
 Equation \eqref{pioutrrr} tells us that visiting a lot of stretches has, in average, a strong energetic cost, and that therefore these trajectories
 do not contribute a lot to the partition function
 (this is formalized in the proof of Lemmas  
 \ref{durdur} and \ref{durdur2}). In order to have a result that holds almost surely, however, one has to be careful in the way 
 of using Borel-Cantelli Lemma.

 For the second type of trajectories, on the other hand, we observe that
 in order not to visit many $\hat\tau$-stretches,
 one has to put a lot of contacts in very few $\hat \tau$-stretches,
 and this strategy has a large entropic cost (which is a priori not that
 easy to control). The most convenient way of doing this is to target sufficiently large stretches
 and put the contacts there.
 The key idea to estimate this is to realize that in
 order to visit the long stretches without having too many
 contacts before, $\tau$ has to grow much faster
 that  it would typically do, in the sense that $\tau_x$ has to be  larger
 than $x^{\tilde \alpha (1\wedge \alpha)}$ (cf. Lemmas \ref{pasmieu} and 
 \ref{pasmieu2}), which is much larger than 
 what it would typically be, that is, $x^{(1\wedge \alpha)}$.
 We get this thanks to Lemma
 \ref{lem:boundmax} which says that the first  $\hat\tau$-stretch
 of size $l\gg1$ occurs at distance
 approximately $l^{\tilde\alpha}$ from the origin.
 One also notices that targeting at the first jump a sufficiently
 large $\hat\tau$-stretch and putting all the contacts in it already gives the right
 lower bound in Proposition \ref{prop:tail}, and we believe this is the right strategy
 for the polymer to adopt.


\section{Lower bound on the free energy} \label{seclb}

We prove in this section the easier half of Theorem \ref{freeE}. Here and later we choose $h$ small enough 
(then one can say that the results hold for all $h\in (0,1)$ by modifying the constant $C_1$).
For practical reasons we compute a lower bound for $\bar \tf(\gb,h)$ which according to \eqref{altfreen2} 
is equal to $\tf(\gb,h)$ up to a multiplicative constant.
Then, according to \eqref{altfreen}, it is sufficient to estimate 
$\hat \bE\ \bbE [\log Z_{\hat \tau_N,h}^{\go,\gb,\pin}]$ for a given $N$ to get a lower bound.

\medskip

We define $M_{N}$ to be the size of the longest inter-arrival among the $N$ first of the renewal~$\hat \tau$:
\begin{equation}
 M_{N} := \max_{i\in[1,N]} (\hat\tau_{i}-\hat \tau_{i-1}),
\end{equation} 
and $i_{\max}$ to be the smallest index such that $\hat\tau_{i}-\hat \tau_{i-1}=M_N$.
In order to get an explicit lower bound on $Z_{\hat \tau_N,h}^{\go,\gb,\pin}$ we consider the contribution of trajectories $\tau$ that
have contacts with the defect line only in the interval $(\htau_{i_{\max-1}},\htau_{i_{\max}}]$.

\begin{figure}[hlt]

\leavevmode
\epsfxsize =14 cm
\psfragscanon
\psfrag{0}{$0$}
\psfrag{N}{$\hat\tau_N$}
\psfrag{MN}{$M_N$}
\psfrag{T1}{$\hat\tau_{i_{\max-1}}$}
\psfrag{T2}{$\hat\tau_{i_{\max}}$}
\epsfbox{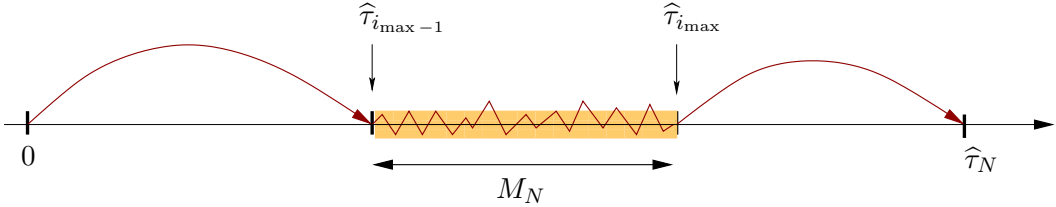} 
\begin{center}
   \caption{The strategy to get the lower bound is to target only the longest $\hat\tau$-stretch, which is
of size $M_N$, starting at $\hat\tau_{i_{\max} -1}$ and ending at $\hat\tau_{i_{\max}}$.}
\end{center}
\label{figlongstretch}
\end{figure}

If $X_{i_{\max}}=0$, then one has
\begin{equation}
Z_{\hat \tau_N,h}^{\go,\gb,\pin}
\ge K(\htau_{i_{\max}-1})e^{h-\gb} Z^{\pin}_{M_N,h} K(\hat \tau_N-\htau_{i_{\max}})e^{h-\gb},
 \end{equation}
where 
$Z^{\pin}_{N,h}$ denotes the partition function of the homogeneous pinning
model with pinned boundary condition (similar to \eqref{pbc} but with 
$\gb=0$).
Now note that our assumptions on $K(\cdot)$ ensures that for $N$ sufficiently large one has
\begin{equation} 
\min( K(\htau_{i_{\max}-1}),  K(\hat \tau_N-\htau_{i_{\max}}))\ge \frac12 \hat c_K (\htau_N)^{-(1+\alpha)}. 
\end{equation}

From all this one gets that there exists a constant $C_3$ (depending on $\beta$) such that
\begin{equation}\label{okeq}
\hat \bE\bbE \left[ \log Z_{\hat \tau_N,h}^{\go,\gb,\pin} \right] \ge \bbP( X_{i_{\max}}=0 )\left(C_3 - 2 (1+\alpha) \hat \bE[\log \htau_N]+ 
 \hat\bE \log Z^{\pin}_{M_N,h}\right).
\end{equation}
Then, one must estimate $\hat \bE \log Z^{\pin}_{M_N,h}$.
We use the following estimate for $Z_N^{\pin}$

\begin{lemma}\label{totti}
There exists a constant $C_4$ such that for every $h\in (0,1)$, and every $N$,
\begin{equation}
 Z^{\pin}_{N,h} \ge C_4 N^{-1}e^{N\tf(h)}.
\end{equation}
\end{lemma}

\begin{proof} 
We first observe that for every pair of integers $(n_1,n_2)$, decomposing over the first return time after $n_1$, one has
\begin{equation}
e^h Z_{n_1+n_2,h} \leq e^h Z_{n_1,h} e^h Z_{n_2,h},
\label{sousmult}
\end{equation} 
so that the sequence $\{\log (e^h Z_{N,h})\}_{N\in\N}$ is subadditive. Then one has that $\tf(h)$
verifies
 $\tf(h) = \inf_{N\in\N} \frac{1}{N} \log e^h Z_{N,h}$ and $Z_{N,h} \geq e^{-h} e^{N \tf(h)}$
for all $N$.
And therefore, one gets the result by using \eqref{houhi} which gives
\begin{equation}
 Z_{N,h}^{\pin}\ge c N^{-1} Z_{N,h}.
\end{equation}
\end{proof}

Plugging the above result into \eqref{okeq} one has
\begin{multline}\label{machin23}
\hat \bE\ \bbE\left[ \log Z_{N,h}^{\go,\gb,\p}\right]\ge \frac{C_3}{2} - (1+\alpha)\hat \bE[\log \hat\tau_N] +
   \frac{1}{2} \hat \bE\left[ M_N \tf(h)-\log M_N \right] + \frac{\log C_4}{2}
 \\ \ge   \frac{1}{2} \hat \bE[ M_N ] \tf(h)-C_5\log \hat \bE[\hat\tau_N]+C_6
        \geq \frac{1}{2} \hat \bE[ M_N ] \tf(h)-C_5\log N -C_7,
\end{multline}
where we used in the second inequality that $M_N\leq \hat\tau_N$ and Jensen inequality so that $C_5=\frac32+\alpha$,
and in the second one that $\hat \bE[\hat\tau_N] = N\hat\bE[\hat \tau_1]$ so that $C_7 = C_5\log \hat\bE[\hat \tau_1] -C_6$.
From the assumption we have on $\hat K$, one has, uniformly for all $n\gg N^{\gep}$,
\begin{equation}
 \hat \bP \left[M_N\le n\right]= \hat\bP(\htau_1\le n)^N = \exp\left(-\frac{\hat c_K}{\tilde \alpha} N n^{-\tilde \alpha}(1+o(1))\right).
\end{equation}
So that using Rieman sum as approximation of integral one gets that
 $\hat \bE[M_N]=(C_8+o(1))N^{\tilde \alpha^{-1} }$,
where 
\begin{equation}
 C_8= \int_{0}^{\infty} \left(1- \exp\left(-\frac{\hat c_K}{\alpha}x^{-\tilde \alpha} \right) \right).
\end{equation}
Now we choose $N$ to be equal to $N_h:=C_{9} h^{-\frac{\tilde \alpha}{(1\wedge \alpha)}}|\log h|^{\tilde \alpha}$, so that
if $h$ is small enough
\begin{equation}
   \frac12 \tf(h)\hat \bE [M_{N_h}]-C_6\log N_h\ge \frac{C_7}{2}\tf(h)N_h^{1/\tilde \alpha}-C_6 \log N_h\ge |\log h|,
\end{equation}
where the last inequality holds provided $C_{9}$ (entering in the definition of $N_h$) is large enough, using the behavior of $\tf(h)$ as $h$ goes to $0$.
This combined with \eqref{machin23} gives the lower inequality in \eqref{fre1} as
\begin{equation}
 \tf(\gb,h)\ge \frac{1}{N_h}\hat\bE\ \bbE\left[\log Z_{N_h,h}^{\go,\gb,\pin}\right].
\end{equation}

\section{Upper bound on the free energy when $\alpha>1$}\label{prr}

The next two sections are devoted to the proof of the upper bound for the free-energy.
This is much more complicated than the lower bound, as one has to control the contribution of all possible trajectories for $\tau$.

\smallskip

Somehow, things get technically simpler if one does not try to capture the $(\log h)^{1-\tilde \alpha}$ factor. 
Therefore we prove first a rougher result, to give a clear presentation of the strategy we use.
For the two next sections, we use the alternative construction for the environment $\go$ based on the renewal $\tilde \tau$ and presented in 
equation \eqref{altern}.

For this section we introduce the following notation
\begin{equation}\begin{split}
 \tilde T_n&= \tilde \tau_{2n},\ \forall n\ge 0,\\
 \xi_n&=  \tT_{n}- \tT_{n-1}, \ \forall n\ge 1.
\end{split}\end{equation}

\subsection{Rough bound}\label{tsw}
\begin{proposition}
 When $\alpha>1$, one can find a constant $C_2$ such that
\begin{equation}
 \tf(\gb,h)\le  C_2 h^{\tilde \alpha}.
\end{equation}
 \end{proposition}

\begin{proof}
The idea of the proof is to say that only the long stretches of $\go$ with $\go\equiv 0$ can contribute to the free energy and that others cannot.
The first step is to perform a kind of coarse-graining procedure in order to treat the contribution of each segment $(T_n,T_{n+1}]$,
separately (Lemma \ref{lem1} below), and then to show that the contribution of segments that are too short is zero.
 
\smallskip

It turns out that the coarse graining we present here is not optimal and this is the reason why a $\log $ factor is lost. An improved coarse graining
method is presented in the next subsection.

We introduce a new notation to describe the contribution of a given segment:
for $a$ and $b \in \bbN$, one defines (recall that $\theta$ is the shift operator defined just before \eqref{aftshd})
\begin{equation}
 Z^{\go,\gb}_{[a,b],h}:= \exp(\gb\go_a+h)Z_{(b-a),h}^{\theta^a \go,\gb}.
\end{equation}
Here is our coarse graining Lemma

\begin{lemma} \label{lem1}
For every $N\in \N$ 
\begin{equation}
 Z_{\tT_N,h}^{\go,\gb}\le \prod_{i=1}^N \left[\left(\max_{x\in (\tT_{i-1}, \tT_{i}]}
Z^{\go,\gb}_{[x,\tT_{i}],h}\right)\vee 1\right].
\end{equation}
\end{lemma}

\begin{proof}
We proceed by induction.
The claim is obvious for $N=1$. For the process $\tau$ define $\tau_{\mathrm{next}}^{(N)}:=\inf\{n>\tT_N, n\in \tau\}$ then one has
(using the Markov property for $\tau$)
\begin{multline}
  \frac{ Z_{\tT_{N+1},h}^{\go,\gb}}{ Z_{\tT_N,h}^{\go,\gb}}=\bE^{\go,\gb}_{\tT_{N},h}
\left[\exp\left(\sum_{n=\tT_{N}+1}^{\tT_{N+1}}(\gb\go_n+h)\ind_{\{n\in \tau\}}\right)\right]\\
=\sum_{x=\tT_{N}+1}^{\tT_{N+1}}\bP^{\go,\gb}_{\tT_{N},h}\left(\tau_{\mathrm{next}}^{(N)}=x\right) 
Z^{\go,\gb}_{[x,\tT_{N+1}],h}+\bP^{\go,\gb}_{\tT_{N},h}\left(\tau_{\mathrm{next}}^{(N)}>\tT_{N+1}\right). 
\end{multline}
And the above sum is smaller than $\left(\max_{x\in (\tT_{N}, \tT_{N+1}]}
Z^{\go,\gb}_{[x,\tT_{N+1}],h}\right)\vee 1$ as it is a convex combination of the terms in the maximum.
\end{proof}

Now we remark that by definition $\go_{\tT_i}=-1$. Therefore, for any $x\in (\tT_{i-1}+1, \tT_{i}]$ one has
\begin{multline}
Z^{\go,\gb}_{[x,\tT_i],h}=\bE\left[ e^{\sum_{n=x}^{\tT_i} (\gb\go_{n+x}+h)\ind_{\{n\in \tau\}}}\right]\le e^{h(\tT_i-x)}
 \bE\left[e^{\gb\go_{\tT_i} \ind_{\{\tT_i-x\in \tau\}}}\right] \\ 
=e^{h (\tT_{i}-x)} \left[1-(1-e^{-\gb})\bP(\tT_{i}-x\in \tau)\right] \\
\le  e^{h\xi_i} \left(1-(1-e^{-\gb})\inf_{n \ge 1}\bP(n\in \tau)\right).
\end{multline}
As $\bE[\tau_1]<\infty$, the renewal Theorem \cite[Chapter 1, Theorem 2.2]{Asmuss}
ensures that $\inf_{n \ge 1}\bP(n\in\tau)>0$.
From this one obtains the following result that we record as a lemma
\begin{lemma}\label{etap2}
 One can find a constant $C_{10}>0$ (depending on $\gb$) such that the following bounds hold
\begin{equation}
\begin{split}
\max_{x\in (\tT_{i-1}, \tT_{i}]}
Z^{\go,\gb}_{[x,\tT_{i}],h}\le (1-C_{10}) \hspace{1.2cm}& \quad \text{ if } 
\xi_i<  C_{10} h^{-1},\\
\max_{x\in (\tT_{i-1}, \tT_{i}]}
Z^{\go,\gb}_{[x,\tT_{i}],h}\vee 1 \le e^{h\xi_i} &\quad \text{ if } 
\xi_i\ge C_{10}h^{-1}.
\end{split}
\label{eq:differentZ}
\end{equation}
\end{lemma}

Then, the only segments that contribute to the free energy are the segments longer than $C_{10}h^{-1}$.
From Lemma \ref{lem1} and \ref{etap2} one gets that
\begin{equation}
 \log Z_{\tT_N,h}^{\go,\gb}\le h \sum_{i=1}^N  \xi_i \ind_{\{\xi_i> C_{10} h^{-1}\}}. 
\end{equation}
Now using (twice) the law of large numbers one gets that
\begin{equation}
\tf(\gb,h)=\lim_{N\to \infty} \frac{N}{\tT_N}\frac{1}{N}\log Z_{\tT_N,h}^{\go,\gb}\le \frac{1}{\tilde \bE \left[\xi_1\right]}
h \tilde \bE\left[\xi_1 \ind_{\{\xi_1> C_{10} h^{-1}\}}\right].
\end{equation}
From the definition of $\xi$ and the properties \eqref{tidd} of the renewal $\tilde \tau$ one gets that
$\tilde \bE\left[ \xi_1\right]$ is a positive constant, and that
\begin{equation}
 \bE\left[\xi_1 \ind_{\{\xi_1> C_{10} h^{-1}\}}\right]\le C_{11} h^{\tilde \alpha-1}. 
\end{equation}
This finishes the proof.
\end{proof}

\subsection{Finer bound}
\label{sec:finer1}

The reason why we lose a power of $\log h$ in the previous proof is that our coarse graining Lemma does not take into account the cost
for $\tau$ to do long jumps between the segments contributing to the free energy.
We present in this section a method to control this. This is rather technical but allows to get an upper bound matching 
the lower bound proved in Section \ref{seclb}.

\begin{proposition}\label{alphapgq1}
 When $\alpha>1$, one can find a constant $C_2$ such that
\begin{equation}
 \tf(\gb,h)\le  C_2 h^{\tilde \alpha} |\log h|^{1-\tga}
\end{equation}
 \end{proposition}
\begin{proof}
We define the sequence $(J_i)_{i\ge 0}$ as $J_0:=0$, and
\begin{equation}
 J_{i+1}:=\inf\{n > J_i, \xi_{n+1}\ge C_{10} h^{-1}\},
\end{equation}
with the constant $C_{10}$ given in Lemma \ref{etap2}.
Furthermore one sets
\begin{equation}
 \cT_N:=\tilde T_{J_N}. 
\end{equation}
We have cut the system in \textit{metablocks} composed
of one block bigger than $C_{10} h^{-1}$, and then other smaller blocks.
As the free-energy is a limit in the almost sure sense, 
conditioning to an event of positive 
probability (for the environment) is harmless. For matters of translation invariance
 (we want the sequence $\left\{(\go_n)_{n\in(\cT_N,\cT_{N+1}]}\right\}_{N\ge0}$ to be i.i.d.) 
we choose to observe an environment conditioned to 
satisfy $\xi_1\ge C_{10} h^{-1}$.  We denote this conditioned probability by $\tilde \bP^{(1)}$.

\begin{figure}[htbp]
\centerline{
\psfrag{0}{$0$}
\psfrag{xi1}{$\xi_1$}
\psfrag{xi2}{$\xi_2$}
\psfrag{xi3}{$\xi_3$}
\psfrag{xiJ1}{$\xi_{J_1}$}
\psfrag{xiJ11}{$\xi_{J_1+1}$}
\psfrag{T1}{$\tT_1$}
\psfrag{T2}{$\tT_2$}
\psfrag{TJ1}{$\cT_1=\tT_{J_1}$}
\psfrag{TJ11}{$\tT_{J_1+1}$}
\psfig{file=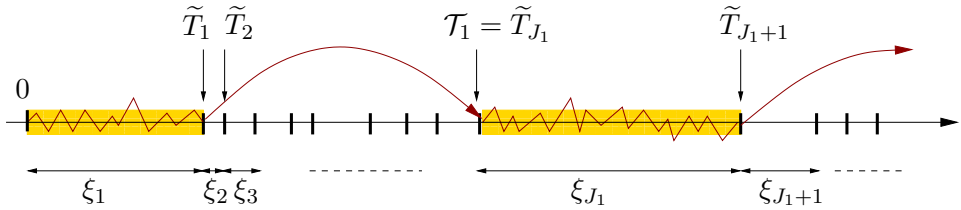,width=5in} }
  \begin{center}
 \caption{ The above figure represents the decomposition of our environment according to the metablocks $(\cT_{i-1},\cT_i]$,
constituted of $J_i-J_{i-1}$ unit blocks entities.
A metablock is composed of unit block larger than $C_{10}h^{-1}$ followed by a 
sequence of $J_{i}-J_{i-1}-1$ (possibly equal to zero but that is quite rare) smaller unit blocks.
As we explained in Section \ref{sec:strategy}, the trajectory of the polymer targets the blocks with $\xi_i \geq h^{-1}|\log h|$.
Our proof, and in particular Lemma \ref{tosto} confirms this idea. It also says that regions of smaller length but located close to each other
could possibly contribute to the free energy, but the quantitative estimates in
equation \eqref{contcomp} show that this contribution is negligible.}   
  \end{center}
\label{figmetablocks}
\end{figure}

In analogy with Lemma \ref{lem1} one has the following decomposition for the partition function
\begin{equation}\label{finerdeco}
 Z^{\go,\gb}_{\cT_N,h}\le \prod_{i=1}^N \left(\max_{x\in (\cT_{i-1}, \cT_{i}]}
Z^{\go,\gb}_{[x,\cT_{i}],h}\right)\vee 1.
\end{equation}
(the proof being exactly the same). 
This allows to treat the contribution to  $Z^{\go,\gb}_{\cT_N,h}$ of the different segments $(\cT_i,\cT_{i+1}]$ separately. 

Now what we show is that the segment $(\cT_{i}, \cT_{i+1}]$ gives a contribution to the free energy only 
if one of the two following condition is satisfied:
\begin{itemize}
 \item $\xi_{J_{i}+1}$ is much larger than $C_{10} h^{-1}$ (by a factor $|\log h|$),
 \item $J_{i+1}-J_i$ is unusually small.
\end{itemize}
In the other cases, we show that
the energy gain that one has on the block $(\cT_{i},\tT_{J_i+1}]$ is overcome by
the entropic cost of touching the defect line on the segment $(\tT_{J_i+1},\cT_{{i+1}}]$.

\begin{lemma} \label{tosto}
For any $n \ge 0$, any $\delta>0$ there exists a constant $C_{12}$ depending on $\gb$ and $\delta$ such that 
if $\xi_{J_{n}+1} < C_{12} h^{-1} |\log h|$ and $J_{n+1}-J_{n} \ge h^{-1-\delta} $, then 
\begin{equation}\label{youhou}
 \max_{x\in (\cT_{n}, \cT_{n+1}]}
Z^{\go,\gb,}_{[x,\cT_{n+1}]}\le 1.
\end{equation}
If  $\xi_{J_{n}+1} \ge C_{12} h^{-1} |\log h|$ or $J_{n+1}-J_{n} \le h^{-1-\delta} $
then
\begin{equation}\label{syouhou}
 \max_{x\in (\cT_{n}, \cT_{n+1}]}
Z^{\go,\gb}_{[x,\cT_{n+1}]}\le e^{h \xi_{J_n+1}}.
\end{equation}

\end{lemma}

We postpone the proof of the Lemma to the end of the section and prove Proposition~\ref{alphapgq1} now.

Combining Lemma \ref{tosto} and the decomposition \eqref{finerdeco} one gets that

\begin{equation}
 \log Z^{\go,\gb}_{\cT_N,h}\le h \sum_{n=0}^{N-1}  \xi_{J_n+1} 
\ind_{\{\xi_{J_{n}+1}\ge C_{12}h^{-1}|\log h| \text{ or } J_{n+1}-J_n \le h^{-1-\delta}\}}
\end{equation}
Note that
the terms in the sum of right-hand side are i.i.d.\ distributed and have finite mean.
Therefore using twice the law of large numbers, one gets
\begin{multline}
 \tf(\gb,h)\leq\lim_{N\to \infty} \frac{N}{\cT_N}\frac{1}{N}h \sum_{n=0}^{N-1}  \xi_{J_n+1} 
\ind_{\{\xi_{J_{n}+1}\ge C_{12}h^{-1}|\log h| \text{ or } J_{n+1}-J_n \le h^{-1-\delta}\}}
\\ = \frac{h}{\tilde \bE^{(1)}\left[\cT_1\right]} \tilde \bE^{(1)}\left[ \xi_{1} 
\ind_{\{\xi_{1}\ge C_{12} h^{-1}|\log h| \text{ or } J_1 \le h^{-1-\delta}\}}\right].
\end{multline}
From its definition one has
\begin{equation}
 \tilde \bE^{(1)}\left[\cT_1\right]=\frac{\tilde \bE[\xi_1\ind_{\{\xi_1\ge C_{10} h^{-1}\}}]}{\tilde \bP[\xi_1\ge C_{10}h^{-1}]}+ \tilde \bE[J_1-1]\frac{\tilde 
\bE\left[\xi_1\ind_{\{\xi_1< C_{10} h^{-1}\}}\right]}{\tilde \bP[\xi_1< C_{10}h^{-1}]}
=\frac{\tilde \bE[\xi_1]}{\tilde \bP\left[\xi_1\ge C_{10}h^{-1}\right]},
\label{estimcT1}
\end{equation}
 where the last equality comes from the fact that $J_1$ is a geometric variable of parameter $\tilde\bP(\xi_1\geq C_{10}h^{-1})$.
It remains to estimate
\begin{multline}
 \tilde \bE^{(1)}\left[ \xi_{1} 
\ind_{\{\xi_{1}\ge C_{12}h^{-1}|\log h| \text{ or } J_1 \le h^{-1-\delta}\}}\right]\\
\le
\tilde \bE^{(1)}\left[ \xi_{1} \ind_{\{\xi_{1}\ge C_{12} h^{-1}|\log h|\}}\right]+
  \tilde \bE^{(1)}\left[ \xi_{1} \ind_{\{ J_1 \le h^{-1-\delta}\}}\right].
\end{multline}
 The first term gives the main contribution, it is equal to
\begin{equation}
 \frac{\tilde \bE\left[\xi_1 \ind_{\{\xi_1\ge C_{12}h^{-1}|\log h|\}}\right]}{\tilde \bP\left[\xi_1\ge C_{10} h^{-1}\right]}.
\end{equation}
The second one is equal to 
\begin{equation}
 \frac{\tilde \bE\left[\xi_1 \ind_{\{\xi_1\ge C_{10}h^{-1}\}}\right]}{{\tilde \bP\left[\xi_1\ge C_{10} h^{-1}\right]}}
 \tilde \bP\left[ J_1\le h^{-1-\delta}\right],
\end{equation}
so that overall 
\begin{multline}
 \tf(\gb,h)\le h(\tilde\bE[\xi_1])^{-1}\\
\left(\tilde \bE\left[\xi_1 \ind_{\{\xi_1\ge C_{11}h^{-1}|\log h|\}}\right]+
\tilde \bE\left[\xi_1 \ind_{\{\xi_1\ge C_{10}h^{-1}}\}\right]\tilde \bP\left[ J_1\le h^{-1-\delta}\right]\right).
\end{multline}
Then one can check, using \eqref{tidd}, that there exists $C_{13}$ such
that
\begin{equation}\label{contcomp}
\begin{split}
\tilde \bE\left[\xi_1 \ind_{\{\xi_1\ge C_{10}h^{-1}\}} \right] &\leq C_{13} h^{\tga-1}, \\
\tilde \bE[ \xi_1\ind_{\{\xi_1\ge C_{13} h^{-1}|\log h|\}} ]&\le C_{13} |\log h|^{1-\tilde \alpha} h^{\tilde \alpha -1},\\
\tilde \bP[ J_1 \le h^{-1-\delta} ]&\leq h^{-1-\gd} \tilde\bP(\xi_1\ge C_{10}h^{-1})\le  C_{13}  h^{\tilde \alpha-1-\delta},
\end{split}
\end{equation}
which is enough to conclude.
\end{proof}

\begin{proof}[Proof of Lemma \ref{tosto}]
We start by remarking that by translation invariance (from our choice to impose that $\xi_1\geq C_{10}h^{-1}$) 
it is sufficient to prove the result in the case $n=0$.

We have to control the value of   $Z^{\go,\gb}_{[x,\cT_{1}],h}$ for every $x\in(0,\cT_{1}]$.
We start with the easier case $x>\tT_{1}$. 
In that case we can use the strategy of the previous section:
supposing that $x\in (\tT_{a},\tT_{a+1}]$ then one has (exactly like in the proof of Lemma \ref{lem1}),
\begin{equation}\label{decoco}
 Z^{\go,\gb}_{[x,\cT_1],h}\le \prod_{i=a+1}^{J_{1}}\left[ \left(\max_{y\in (\tT_{i-1}, \tT_{i}]}
Z^{\go,\gb}_{[y,\tT_{i}],h}\right)\vee 1\right].
\end{equation}
and one can show that all the terms in the product on the right hand-side are equal to one, since all blocks
$[\tT_{i-1},\tT_i]$ are smaller than $C_{10}h^{-1}$ (cf.\ Lemma \ref{etap2}).

\smallskip
To prove \eqref{syouhou} one also uses equation \eqref{decoco}, and then Lemma \ref{etap2} 
to bound the different factors of the product on 
the right-hand side. 

\smallskip
Now we turn to the case $x\in(0,  \tT_{1}]$, $\xi_{1}\le C_{11} h^{-1}|\log h|$, $J_1\ge h^{-(1+\delta)}$.
We use the following refinement of our block decomposition
\begin{lemma}\label{parigolo}
For any  $x\in (0,  \tT_{1}]$ there exists a constant $C_{14}\in(0,1)$ (depending on $\beta$, but not on $C_{12}$) such that 
\begin{equation}
 Z^{\go,\gb}_{[x,\cT_1],h}\le e^{\xi_{1} h} \prod_{i=2}^{J_{1}} 
\left( 1- C_{14}\frac{\xi_{i}}{\tT_{i-1}}\right).
\end{equation}
\end{lemma}

\begin{proof}
For notational convenience we also restrict to the case 
$x=0$, but the proof works the same for all values of $x$.

We prove by induction on $j$, that for any $j\in [1, J_{1}]$,
\begin{equation}
 Z^{\go,\gb}_{\tT_{j},h}\le \exp(\xi_{1} h) \prod_{i=2}^{j} \left( 1- C_{14}\frac{\xi_{i}}{\tT_{i-1}}\right).
\end{equation}
 The case $j=1$ is just the second point of Lemma \ref{etap2}. Then for the induction step one remarks that
\begin{equation}
 \frac{Z^{\go,\gb}_{\tT_{j+1},h}}{Z^{\go,\gb}_{\tT_j,h}}=
\bE^{\go,\gb}_{\tT_j,h}
\left[\exp\left(\sum_{k=\tT_{j}+1}^{\tT_{j+1}}(\gb\go_n+h)\ind_{\{k\in \tau\}}\right)\right].
\end{equation}
Define
\begin{equation}\begin{split}
 \tau_{\mathrm{prev}}^{(j)}&:= \max\{\tau_k \ | \ \tau_k \le \tT_{j}\}, \\
 \tau_{\mathrm{next}}^{(j)}&:= \min\{\tau_k \ | \ \tau_k > \tT_{j}\}.
\end{split}
\end{equation}
One can notice that the distribution of $ \tau_{\mathrm{next}}^{(j)}$ knowing  $\tau_{\mathrm{prev}}^{(j)}$ under  $\bP^{\go,\gb}_{\tT_j,h}$
does not depend on $\go$ nor $\gb$ and that one has
(recall $\bar K(n):=\bP(\tau_1>n)$)
\begin{equation}
 \bP^{\go,\gb}_{\tT_j,h}( \tau_{\mathrm{next}}^{(j)}=y | \tau_{\mathrm{prev}}^{(j)}=z)=\frac{K(y-z)}{\bar K(\tT_{j}-z)}.
\end{equation}

Therefore
\begin{multline}
 \bE^{\go,\gb}_{\tT_{j},h}
\left[\exp\left(\sum_{k=\tT_{j}+1}^{\tT_{j+1}}(\gb\go_n+h)\ind_{\{k\in \tau\}}\right)\right]\\=
\sum_{y=\tT_{j}+1}^{\tT_{j+1}}\bP^{\go,\gb}_{\tT_{j},h}\left( \tau_{\mathrm{next}}^{(j)}=y\right) Z^{\go,\gb}_{[y,\tT_{j+1}],h}
+\bP^{\go,\gb}_{\tT_{j},h}\left(\tau_{\mathrm{next}}^{(j)}>\tT_{j+1}\right)\\
\le \max_{z\in [0,\tT_{j}]}\left[ 
\frac{\sum_{t=\tT_{j}+1}^{\tT_{j+1}}K(t-z)}{\bar K(\tT_{j}-z)} 
\max_{y\in (\tT_{j},\tT_{j+1}]} Z^{\go,\gb}_{[y,\tT_{j+1}],h}+
\frac{\bar K(\tT_{j+1}-z)}{\bar K(\tT_{j}-z)}\right].
\end{multline}
From our definitions, we know that $\xi_{j+1}\leq C_{10}h^{-1}$ for all $j\in [1,J_1-1) $,
and therefore Lemma \ref{etap2} gives an upper bound to
the partition functions $Z^{\go,\gb}_{[y,\tT_{j+1}],h}$, for $y\in (\tT_{j},\tT_{j+1}]$.
\begin{multline}
 \bE^{\go,\gb}_{\tT_{j},h}
\left[\exp\left(\sum_{k=\tT_{j}+1}^{\tT_{j+1}}(\gb\go_n+h)\ind_{\{k\in \tau\}}\right)\right]
\le 1-C_{10}\min_{z\in  [0,\tT_{j}]}
\frac{\sum_{t=\tT_{j}+1}^{\tT_{j+1}}K(t-z)}{\bar K(\tT_{j}-z)}.
\label{eq:finlemparigolo}
\end{multline}
From there, we finish the proof by remarking that from our assumption on $K(\cdot)$ (and using the change of variable $z'= \tT_j-z$), 
there exist constants $C_{14}$ and $C_{15}$
such that
\begin{equation}
 \min_{z\in  [0,\tT_{j}]} \frac{\sum_{t=\tT_j+1}^{\tT_{j+1}}K(t-z)}{\bar K(\tT_j-z)} \geq 
C_{15} \min_{z'\in[0,\tT_j]} (z'+1)^{\alpha}\sum_{u=1}^{\xi_{j+1}}(z'+u)^{-(1+\alpha)}\geq C_{14}/C_{10}\frac{\xi_{j+1}}{\tT_j}.
\label{estimmin1}
\end{equation}
where the last inequality comes from a straightforward computation.


%
\end{proof}

We can now finish the proof of Lemma \ref{tosto}.
Note that for all $j\in [2, J_1]$ one has $\xi_{j}/ \tT_{j-1}\le \xi_j/\xi_1 \leq 1$ so that if $C_{14}<1$ one has
\begin{equation}
 \log \prod_{j=2}^{J_1} \left(1- C_{14} \frac{\xi_{j}}{ \tT_{j-1}}\right) \le -C_{14}  \sum_{j=2}^{J_1}\frac{\xi_{j}}{ \tT_{j-1}}.
\end{equation}
Then one remarks that
\begin{equation}
 \sum_{j=2}^{J_1}\frac{\xi_{j}}{ \tT_{j-1}}\ge \sum_{j=2}^{J_1} \sum_{i=\tT_{j-1}+1}^{\tT_j} \frac{1}{i}\ge  \frac{1}{2}\log(\cT_1/\xi_1).
\label{sumxi}
\end{equation}
Given our assumptions $\cT_1\ge J_1\ge h^{-(1+\delta)}$ and $\xi_1\le C_{12} h^{-1} |\log h|$, one has that
$\log(\cT_1/\xi_1)$ is larger than $\frac{\gd}{2}|\log h|$ if $C_{12}$ is small enough.
Then using Lemma \ref{parigolo} one gets that 
\begin{equation}
\log Z_{[x,\cT_1]}^{\go,\gb}\le  (C_{12}-C_{14}\frac{\gd}{4})|\log h|\le 0
\end{equation}
if $C_{12}$ has been chosen small enough.
\end{proof}


\section{Upper bound on the free energy when $\alpha<1$}
\label{conrr}

The case $\alpha<1$ is a bit more difficult than the case $\alpha>1$. The reason is that
one has not $\inf_{n\in \N} \bP(n\in \tau)>0$ (which was really crucial to prove Lemma \ref{etap2})
and one has to  replace this by technical estimates on the renewal  (for example Lemma \ref{lemtecr}) 
that are a bit more difficult to work with.

\smallskip

We have to change the length of the blocks in our coarse graining procedure, and therefore we renew our definition of $\tT$ and $\xi$ 
for this section. Let $C_{16}$ be a fixed (small) constant (how small is to be decided in the proof).
Set $L=L(h):= \lfloor C_{16} h^{-1/\alpha} \rfloor$.

In analogy with the previous section, define 
\begin{equation}\begin{split}
\tilde T_i&:= \tilde \tau_{iL},\ \forall i\geq0, \\
\xi_i&:= \tilde T_{i}-\tilde T_{i-1},\ \forall i\geq1.
\end{split}
\label{newdefTxi}
\end{equation}

As for the case $\alpha>1$ the proof simplifies considerably if one drops the $|\log h|$ factor in the result. We
expose first this simpler proof in the next Section.  Then in Section  \ref{alphafb} 
we refine the argument in order to get the exact upper bound in \eqref{fre1}.

\subsection{Rough bound}

The result we prove in this section is 
\begin{proposition}\label{rabto}
 When $\alpha<1$, one can find a constant $C_2$ such that
\begin{equation}
 \tf(\gb,h)\le  C_2 h^{\frac{\tilde \alpha}{\alpha}}. 
\end{equation}
 \end{proposition}

In order to do so, we prove an asymptotic upper bound for $Z_{\tT_N,h}^{\go,\gb}$.
The first step is  a coarse-graining decomposition of $Z_{\tT_N,h}^{\go,\gb}$ that
allows to treat the contribution of each segment $(\tT_n,\tT_{n+1}]$
separately.
It turns out that we need something a bit more sophisticated than Lemma \ref{lem1}.
\begin{lemma}\label{lem2}
For every $N\in \N$ 
\begin{equation}
  Z^{\go,\gb}_{\tilde T_{N},h} \le \prod_{n=1}^{N} \max_{y\in[0,\tT_n]} \left[
\sum_{x=1}^{\xi_n}\frac{K(x+y)}{\bar K(y)}Z_{[\tT_{n-1}+x,\tT_{n}],h}^{\go,\gb}+\frac{\bar K(\xi_n+y)}{\bar K(y)}\right].
\end{equation}
\end{lemma}

The second ingredient we need is that segments $(\tT_{n-1},\tT_{n}]$ that are short do not contribute to the free energy, or more precisely that only uncommonly 
long segments $(\tT_{n-1},\tT_{n}]$ contribute effectively to the free-energy.
Set $m:=\bE\left[\tilde \tau_1 \right]$.

\begin{lemma}\label{superlem1}
If $\xi_n< 2m L(h)$ then  
\begin{equation}\label{borncool}
 \max_{y\ge 0} \left[\sum_{x=1}^{\xi_n}\frac{K(x+y)}{\bar K(y)}Z_{[\tT_{n-1}+x,\tT_{n}],h}^{\go,\gb}+\frac{\bar K(\xi_n+y)}{\bar K(y)}\right]
         \le 1,
\end{equation}
more precisely there exists a constant $C_{17}>0$ such that for every $y\ge 0$
\begin{equation}\label{bornpamal}
\sum_{x=1}^{\xi_n}K(x+y)Z_{[\tT_{n-1}+x,\tT_{n}],h}^{\go,\gb}\le (1-C_{17})\sum_{x=1}^{\xi_n} K(x+y).
\end{equation}
There exists a constant $C_{18}$ such that if $\xi_{n}\ge2m L(h)$, then
\begin{equation} \label{borntobealive}
 \max_{y\ge 0} \left[\sum_{x=1}^{\xi_n}\frac{K(x+y)}{\bar K(y)}Z_{[\tT_{n-1}+x,\tT_{n}],h}^{\go,\gb}+\frac{\bar K(\xi_n+y)}{\bar K(y)}\right]
 \le e^h Z_{\xi_n,h}\le 
e^{C_{18} h^{1/\alpha} \xi_n}.
\end{equation}
\end{lemma}

\begin{proof}[Proof of Proposition \ref{rabto}].

Combining Lemma \ref{lem2} and Lemma \ref{superlem1} (inequalities \eqref{borncool} and \eqref{borntobealive}), one obtains

\begin{equation}\label{equationfond}
 \log Z^{\go,\gb}_{\tilde T_N,h}\le C_{18} h^{1/\ga} \sum_{n=1}^N \xi_n \ind_{\{\xi_n\ge 2m L(h)\}}.  
\end{equation}
Using (as in the previous sections) twice the law of large numbers one gets that
\begin{equation}
 \tf(\gb,h)\le \frac{1}{\tilde \bE[ \xi_1]} C_{18} h^{1/\ga} \tilde \bE\left[ \xi_1 \ind_{\{\xi_1\ge 2m L(h)\}}\right]. 
\end{equation}
By definition, $\tilde \bE\left[\xi_1\right]=mL(h)$.
Using Proposition \ref{doneybis}, one can estimate
\begin{equation}
 \tilde \bE\left[ \xi_1 \ind_{\{\xi_1\ge 2m L(h)\}}\right]\leq C_{19} \sum_{x=2mL(h)}^{\infty} x L x^{-(1+\tilde\ga)}
\leq C_{20} L^{2-\tilde \alpha}.  
\end{equation}
Replacing $L(h)$ by its value gives the result.
\end{proof}
We turn to the proof of the Lemmata,
\begin{proof}[Proof of Lemma \ref{lem2}]
We prove this once again by induction on $N$. The result is obvious for $N=1$. 
As in Section \ref{prr}, we use the notation
\begin{equation}\begin{split}
\tau_{\mathrm{next}}^{(N)}&:= \min\{\tau_k \ | \ \tau_k > \tT_N\},\\
\tau_{\mathrm{prev}}^{(N)}&:= \max\{\tau_k \ | \ \tau_k \le \tT_N\}       .
                \end{split}
\end{equation}
Decomposing on the different possible values for $\tau_{\mathrm{next}}^{(N)}$ one obtains
\begin{multline}
  \frac{Z^{\go,\gb}_{\tilde T_{N+1},h}}{Z^{\go,\gb}_{\tilde T_{N},h}}=\bE^{\go,\gb}_{\tilde T_{N},h}
\left[\exp\left(\sum_{n=\tilde T_{N}+1}^{\tilde T_{N}}(\gb\go_n+h)\ind_{\{n\in \tau\}}\right)\right]\\
=\sum_{x=\tT_{N}+1}^{\tT_{N+1}}\bP^{\go,\gb}_{\tT_{N},h}\left(\tau_{\mathrm{next}}^{(N)}=x\right) 
Z_{[x,\tT_{N+1}]}^{\go,\gb}+\bP^{\go,\gb}_{\tT_{N},h}\left(\tau_{\mathrm{next}}^{(N)}>\tT_{N+1}\right). 
\end{multline}
Recall that
\begin{equation}
 \bP^{\go,\gb}_{\tT_{N},h}\left(\tau_{\mathrm{next}}^{(N)}=x| \tau_{\mathrm{prev}}^{(N)}=y \right)=\frac{K(x-y)}{\bar K(\tT_N-y)}.
\end{equation}
Taking the maximum over all possibilities for  $\tau_{\mathrm{prev}}^{(N)}$ we have
\begin{equation}
  \frac{Z^{\go,\gb}_{\tilde T_{N+1},h}}{Z^{\go,\gb}_{\tilde T_{N},h}}\le \max_{y\le \tT_N} \left[
\sum_{x=\tT_N+1}^{\tT_{N+1}}\frac{K(x-y)}{\bar K(\tT_N-y)}Z^{\gb,\go}_{[x,\tT_{N+1}],h}+\frac{\bar K(\tT_{N+1}-y)}{\bar K(\tT_{N}-y)}\right],
\end{equation}
and we get the result by making the change of variables $x\to x-\tilde T_N$ and $y\to \tilde  T_N -y$.
\end{proof}

The statement of Lemma \ref{superlem1} is translation invariant; therefore it is enough to prove it for $N=1$.
The core of the proof consists of proving two technical estimates.

\begin{lemma} \label{superlem}
If $\tilde T_1=\xi_1< 2m L(h)$, then one can find $h_0(\gb)>0$ and two constants $C_{21}>0$ and $C_{22}>0$ (depending on $\gb$), such that
for all $h\leq h_0(\gb)$ one has
\begin{equation}
\max_{x\in[0,L(h)/4]} Z^{\go,\gb}_{[x,\tilde T_1],h}\le 1-C_{21},
\end{equation}
and 
\begin{equation}\label{cellela}
 \max_{x\ge L(h)/4} Z^{\go,\gb}_{[x,\tilde T_1],h}\le Z_{2m L(h),h}\le 1+C_{22}.
\end{equation}
where $C_{22}$ can be made arbitrarily small by choosing $C_{16}$ (entering in the definition of $L$) small.
On the contrary $C_{21}$ can be chosen independently of $C_{16}$.
\end{lemma}

\begin{proof}[Proof of Lemma \ref{superlem}]

The second point is standard and we include it here for the sake of completeness.
We notice that
\begin{equation}
\bP\left( |\tau \cap [1,N]| \ge n\right)\le \bP\left( \nexists i\in [1,n],\ \tau_{i}-\tau_{i-1}>N\right)\le (1-\bar K(N))^n.
\end{equation}
Therefore
\begin{multline}
 Z_{N,h}= 1+\sum_{n=1}^N (e^{nh}-e^{(n-1)h})\bP\left[ \tau \cap [1,N] \ge n\right]\\
\le 1+ \sum_{n=1}^N h \left[ e^h \left(1-\bar K(N)\right)\right]^n\le
1+\frac{h}{1-e^{h}(1-\bar K(N))},
\end{multline}
where the last inequality holds only if $e^{h}(1-\bar K(N))<1$.
Now one uses that for $h$ small 
\begin{equation}
1-e^{h}(1-\bar K(N)) \ge 1-(1+2h)(1-\bar K(N))\ge \bar K(N) -2h,
\end{equation}
and also that
$\bar K(N)\geq (2\ga)^{-1} c_K N^{-\ga}$ for $N$ large enough (from the definition of $K(\cdot)$). Then plugging $N=2mL(h)$,
and recalling our definition of $L(h)$, one has
 (for $h$ small enough)
\begin{equation}
 1-e^{h}(1-\bar K(N)) \ge  h \left( (2\ga)^{-1} c_K (2m C_{16})^{-\ga}  -2 \right) .
\end{equation}
Then the result holds, setting $C_{22}:= \left( (2\ga)^{-1} c_K (2m C_{16})^{-\ga}  -2 \right)^{-1} $.

\medskip

The first point is more delicate and we focus on it now. Take $x\leq L/4$, and 
note that $[\ttau_{L/2},\ttau_L]\subset [x,\tilde T_1] $, so that 
\begin{equation}
|\{i \in[x,\tilde T_1] \ | \ \go_i=-1\}\ge L/4.
\end{equation} 

As $\tT_1=\xi_1 \leq 2mL$, this means that the proportion of $\go$ equal to $-1$ in $[x,\tT_1]$ is at least $1/(8m)$.
We use this fact to prove that the renewal $\tau$ starting from $x$ has to hit one of these $-1$ with positive probability.
This is the content of the following Lemma whose proof is postponed at the end of the section.
\begin{lemma}\label{lemtecr}
 There exists some constant $C_{23}>0$ such that for any $M>0$, $a>0$, if one takes $A$ a subset of $[1,M]$ of cardinality at least $a M$, one has
\begin{equation}
 \bP\left( \tau \cap A \neq \emptyset \right)\ge C_{23} a^{1+\ga}.
\end{equation} 
\end{lemma}

Set $a=1/(8m)$, $M=\tilde T_1-x$ and $A:= \{ n\in [1,\tT_1-x]\ |\ \go_{x+n}=-1\}$.
Using translation invariance of $\tau$, one gets
\begin{multline}
  e^{-h} Z_{[x,\tilde T_1],h}^{\go,\gb}\le \bE\left[e^{\sum_{n=1}^{\tilde T_1-x} h \ind_{\{n\in \tau\}}}\ind_{\{\tau \cap A= \emptyset\}} \right]
+e^{-\gb} \bE \left[e^{\sum_{n=1}^{\tilde T_1-x} h \ind_{\{n\in \tau\}}} \ind_{\{\tau \cap A\ne \emptyset\}}\right]\\
\le  Z_{\tT_1-x,h}-(1-e^{-\gb})\bP\left( \tau \cap A\neq \emptyset \right) \leq Z_{2mL,h}-C_{23}(1-e^{-\gb})(8m)^{-(1+\ga)},
\label{conseqlemtecr}
\end{multline}
where in the last line we used Lemma \ref{lemtecr}.
This allows us to conclude using \eqref{cellela}: provided that $C_{22}$ is sufficently small (which is ensured by choosing $C_{16}$ small)
one can take $C_{21}= \frac{C_{23}}{2}(1-e^{-\gb})(8m)^{-(1+\ga)}$, provided also that $h$ is small enough to absorb the $e^h$ factor.
\end{proof}

\begin{proof}[Proof of Lemma \ref{superlem1}] 
 We leave to the reader to check that  \eqref{borncool} is a consequence of \eqref{bornpamal} and focus on the proof of the latter.
For $\tT_1=\xi_1\leq 2mL(h)$ and for any $y\ge 0$, Lemma \ref{superlem} gives us
\begin{equation}
 \sum_{x=1}^{\xi_1}K(x+y) Z_{[x,\tT_1]}^{\go,\gb}\le (1-C_{21}) \sum_{x=1}^{L/4}K(x+y)+ (1+C_{22})\sum_{x=L/4+1}^{\xi_1}K(x+y)
\end{equation}
And therefore \eqref{bornpamal} holds if for all $y\ge 0$
\begin{equation}
 \frac{\sum_{x=1}^{L/4}K(x+y)}{\sum_{x=1}^{\xi_1} K(x+y)}\ge  
\frac{\sum_{x=1}^{L/4}K(x+y)}{\sum_{x=1}^{2mL} K(x+y)}\ge\frac{C_{17}+C_{22}}{C_{21}-C_{17}}.
\end{equation}
The middle term above is bounded away from zero uniformly in $L$ and in $y$. Therefore \eqref{bornpamal} holds if $C_{17}$ and $C_{22}$ are small
enough (and from Lemma \ref{superlem}, one can make $C_{22}$ as small as needed by adjusting $C_{16}$).

\medskip

For \eqref{borntobealive}, first notice that for every value of $y$
\begin{equation}
\sum_{x=1}^{\xi_1}\frac{K(x+y)}{\bar K(y)}Z_{[x,\tT_{1}],h}^{\go,\gb}+\frac{\bar K(\xi_n+y)}{\bar K(y)}
 \le \max_{x\in (0,\xi_1]}Z_{[x,\tT_{1}],h}^{\go,\gb}\le  \max_{x\in (0,\xi_1]}Z_{[x,\tT_{1}],h}\le e^h Z_{\xi_1,h}
 \end{equation}
 which gives the first inequality.

Then from equation \eqref{cellela} one has that $e^h Z_{2mL,h}$ is bounded above by a constant,
so that one can write $ e^h Z_{2mL,h} \le e^{\frac{C_{18}}{2} h^{1/\ga} 2mL}$, choosing $C_{18}$ sufficiently large.
Then using the observation \eqref{sousmult}, one has that for every pair of integers $(n_1,n_2)$
\begin{equation}
 e^h Z_{n_1+n_2,h}\le e^h Z_{n_1,h} e^h Z_{n_2,h},
\end{equation}
 which allows us to say  that for every $k\in \N$
\begin{equation}
  e^h Z_{2mkL,h} \le e^{k\frac{C_{18}}{2} h^{1/\ga} 2mL},
\end{equation}
so that (by monotonicity of $Z_{N,h}$ in $N$), \eqref{borntobealive} holds for every $\xi_1\ge 2mL$.
\end{proof}

\begin{proof}[Proof of Lemma \ref{lemtecr}]
First notice that
\begin{equation} 
\bP\left( \tau \cap A\ne  \emptyset \right)\geq \sum_{n=1}^{(aM)^{\alpha}} \bP\left(\tau_{n-1}\le aM/2, \tau_{n}\in A \cap(aM/2,M] \right).
\end{equation}
Now for every $x\le aM/2$, one has
\begin{multline}
  \bP\left(\tau_{n}\in A\cap(aM/2,M]\ |\  \tau_{n-1}=x\right)=\sum_{y\in A\cap(aM/2,M]} K(y-x) \ge \\
|A\cap (aM/2,M]| \min_{m\le M} K(m)\ge 
  \frac{aM}{2} C_{24} M^{-(1+\alpha)},
\end{multline}
and therefore
\begin{equation}
 \bP\left( \tau \cap A\ne  \emptyset \right)\geq \frac{a}{2} M^{-\ga}C_{24}  \sum_{n=1}^{(aM)^{\alpha}} \bP[\tau_{n-1}\le aM/2].
\end{equation} 
As $\bP[\tau_{n-1}\le aM/2] $ is bounded away from zero  uniformly for all $n\le (aM)^{\alpha}$ (see for example (1.8) in \cite{Doney}),
on can find $C_{23}$ such that
\begin{equation}
 \bP\left( \tau \cap A\ne  \emptyset \right)\ge C_{23} a^{1+\alpha}.
\end{equation}
\end{proof}

\subsection{Finer bound}\label{alphafb}

As in Section \ref{prr}, to get the $|\log h|^{1-\tga}$ factor, one needs a new coarse graining procedure which takes into account
the cost for $\tau$ of doing long jumps between blocks that effectively contribute to the free energy.
We are then able to get an upper bound on the free energy that matches the lower bound proved in Section \ref{seclb}.

\begin{proposition}
 When $\ga<1$, one can find a constant $C_2$ such that
\begin{equation}
 \tf (\gb,h)\leq C_2 h^{\tga/\ga} |\log h|^{1-\tga}.
\end{equation} 
\label{alphappq1}
\end{proposition}

The method is quite similar to
the one used in the case $\ga>1$.
Define the sequence $(J_i)_{i\geq 0}$ as
$J_0:=0$, and
\begin{equation}
 J_{i+1}:= \inf \{n>J_i, \xi_{n+1}\geq 2mL(h)\}.
\end{equation} 
Set $\cT_N:=\tT_{J_N}$.
Note that we used for $\tT_i$ and $\xi_i$ the definitions \eqref{newdefTxi}.

Our system is decomposed in metablocks made of one block bigger than $2mL$, and then other smaller blocks. This is the
same type of decomposition as shown in Figure \ref{figmetablocks}, except that the blocks that constitute one metablock
are already composed of $L$ $\hat\tau$-jumps (instead of $2$ in the case $\ga>1$), so that their typical size is $mL$.

\medskip

We proceed as in Section \ref{sec:finer1}, conditioning the environment to satisfy $\xi_1\geq 2mL$. We denote
this conditioned probability $\tilde \bP^{(1)}$, and underline that as far as the free energy is concerned,
conditioning the environment to an event of
positive probability is harmless. This is done for a matter of translation invariance: thanks to this trick
the sequence $\{(\go_n)_{n\in (\cT_N,\cT_{N+1}]}\}_{N\geq 0}$ is i.i.d. under $\bP^{(1)}$.
\smallskip

As we did in Lemma \ref{lem2}, we can get an upper bound on the free-energy that factorizes the contribution of the different blocks 
\begin{equation}
Z_{\cT_N,h}^{\go,\gb} \leq \prod_{n=0}^{N-1} \max_{y\in [0,\cT_{n+1}]}
  \left[ \sum_{x=1}^{\cT_{n+1}-\cT_{n}} \frac{K(x+y)}{\bar K(y)} Z_{[\cT_{n}+x,\cT_{n+1 }],h}^{\go,\gb}+
      \frac{\bar K(\cT_{n+1}-\cT_{n}+y)}{\bar K(y)}\right].
\label{newdecomp}
\end{equation} 
The proof being exactly the same that for Lemma \ref{lem2}, we leave it to the reader (we will use this kind of coarse graining repeatedly in the remaining of the paper).

\medskip

Now, we show a Lemma analogue of Lemma \ref{tosto}, which tells that a block $(\cT_i,\cT_{i+1}]$
contributes to the free energy only if $\xi_{J_i+1}$ is much larger than $2mL$ (by a factor $\log L$), or if $J_{i+1}-J_i$ is
relatively small.

\begin{lemma}
 \label{tosto2}
There exists a constant $C_{16}$ (entering in the definition of $L(h)$), such that for any $n\geq 0$:\\
If $\xi_{J_n+1} < L \log L$ and $J_{n+1}-J_n\geq L^{(\tga+1)/2}$, then
\begin{equation}\label{tostimportant}
 \max_{y\geq 0}
  \left[ \sum_{x=1}^{\cT_{n+1}-\cT_n} \frac{K(x+y)}{\bar K(y)} Z_{[\cT_n+x,\cT_{n+1}],h}^{\go,\gb}+ 
\frac{\bar K(\cT_{n+1}-\cT_n+y)}{\bar K(y)}\right] = 1.
\end{equation} 
If $\xi_{J_n+1} \geq L \log L$ or $J_{n+1}-J_n < L^{(\tga+1)/2} $, then
\begin{equation}\label{tostmineur}
 \max_{y\geq 0}
  \left[ \sum_{x=1}^{\cT_{n+1}-\cT_n} \frac{K(x+y)}{\bar K(y)} Z_{[\cT_n+x,\cT_{n+1}],h}^{\go,\gb}+
       \frac{\bar K(\cT_{n+1}-\cT_n+y)}{\bar K(y)}\right] \leq e^{C_{18}h^{1/\ga} \xi_{J_n+1}}.
\end{equation} 
(For the same constant $C_{18}$ as in Lemma \ref{superlem1}).
\end{lemma}

We postpone the proof of the Lemma to the end of the section.
\begin{proof}[Proof of Proposition \ref{alphappq1}]
 From the decomposition \eqref{newdecomp} and Lemma \ref{tosto2}, one has
\begin{equation}
 \log Z_{\cT_N,h}^{\go,\gb} \leq C_{18} h^{1/\ga} \sum_{n=0}^{N-1} \xi_{J_n+1}
     \ind_{\{\xi_{J_n+1} \geq L \log L \text{ or } J_{n+1}-J_n < L^{(\tga+1)/2}\}}.
\end{equation}
Using twice the law of large numbers one gets as a consequence
\begin{equation}
 \tf (\gb,h)\leq \frac{C_{18}}{\tilde\bE^{(1)}[\cT_1]}h^{1/\ga}
  \tilde\bE^{(1)} \left[\xi_1 \ind_{\{\xi_{1} \geq L \log L \text{ or } J_1 < L^{(\tga+1)/2}\}}  \right]. 
\end{equation} 
Then in analogy with \eqref{estimcT1}, one gets from the definition of $\cT_1$ that
\begin{equation}
 \tilde\bE^{(1)}[\cT_1] = \frac{\tilde \bE[\xi_1\ind_{\{\xi_1\ge 2mL\}}]}{\tilde \bP[\xi_1\ge2mL]}+ \tilde \bE[J_1-1]\frac{\tilde 
\bE\left[\xi_1\ind_{\{\xi_1< 2mL\}}\right]}{\tilde \bP[\xi_1<2mL]}
= \frac{\tilde\bE[\xi_1]}{ \tilde \bP[\xi_1 \geq 2mL]}.
\end{equation}
One also has
\begin{multline}
 \tilde\bE^{(1)} \left[\xi_1 \ind_{\{\xi_{1} \geq L \log L \text{ or } J_1 < L^{(\tga+1)/2}\}}  \right] \\
  \leq   \tilde\bE^{(1)} \left[\xi_1 \ind_{\{\xi_{1} \geq L \log L\}} \right] +
       \tilde\bE^{(1)}  \left[ \xi_1 \ind_{\{ J_1 < L^{(\tga+1)/2}\}}  \right]\\
  = \frac{\tilde\bE \left[\xi_1 \ind_{\{\xi_{1} \geq L \log L\}} \right]}{\tilde \bP(\xi_1 \geq 2mL)}
 + \frac{\tilde\bE \left[  \xi_1 \ind_{\{\xi_{1} \geq 2mL \}}\right] \tilde\bP (J_1<L^{(\tga+1)/2})}{\tilde \bP(\xi_1 \geq 2mL)},
\end{multline}
and hence
\begin{equation}
 \tf (\gb,h)\leq (mL)^{-1} h^{1/\ga} \left(  \tilde\bE \left[\xi_1 \ind_{\{\xi_{1} \geq L \log L \}}\right]
 + \tilde\bE \left[  \xi_1  \ind_{\{\xi_{1} \geq 2mL \}}\right] \tilde \bP (J_1<L^{(\tga-1)/2})\right),
\end{equation} 
where we also used that $\tilde\bE[\xi_1]=mL$.
Then Proposition \ref{prop:doney} allows us to bound the right-hand side of the above equation: one can check that there exists a constant $C_{25}$ such that
\begin{equation}
 \begin{split}
  \tilde\bE \left[  \xi_1  \ind_{\{\xi_{1} \geq 2mL \}}\right] & \leq C_{25} L^{2-\tga},\\
  \tilde\bE \left[\xi_1 \ind_{\{\xi_{1} \geq L \log L \}}\right] &\leq C_{25} L^{2-\tga} (\log L)^{1-\tga}, \\
  \tilde \bP (J_1<L^{(\tga+1)/2})& \leq  L^{(\tga+1)/2} \tilde\bP(\xi_{1} \geq 2mL)   \leq C_{25} L^{(1-\tga)/2},
 \end{split}
\end{equation} 
which is enough to conclude, recalling the definition of $L(h)$.
\end{proof}

\begin{proof}[Proof of Lemma \ref{tosto2}]

By using translation invariance it is sufficient (and notationally more convenient) to prove the result only in the case $n=0$.

We first prove that in all cases 
\begin{equation}\label{facil}
 \sum_{x\in(\tT_{1},\cT_1]} K(x+y)  Z_{[x,\cT_1],h}^{\go,\gb}\le  \sum_{x\in(\tT_{1},\cT_1]} K(x+y), 
\end{equation}
which is the easy part and then prove that, for every $x\in (1,\xi_1]$
\begin{equation}\label{pludur}
\begin{split}
 Z_{[x,\cT_1],h}^{\go,\gb}&\le 1 \quad \text{when $\xi_1< L \log L$ and $J_1\geq L^{(\tilde\ga+1)/2}$},\\
 Z_{[x,\cT_1],h}^{\go,\gb}&\le e^{C_{18} \xi_1 h^{\frac{1}{\alpha}}} \quad \text{in every other cases}.
\end{split}
\end{equation}
Combining of \eqref{facil}, \eqref{pludur} we prove both \eqref{tostimportant} and \eqref{tostmineur}.

If $x\in(\tT_{a},\tT_{a+1}]$ with $a\in \{1,\ldots,J_1-1\}$, one uses a coarse graining
argument similar to the one of Lemma~\ref{lem2} to factorize
 $Z_{[x,\cT_1],h}^{\go,\gb}$, and also equation \eqref{borncool} in Lemma \ref{superlem1}
to show that since all blocks we consider are of size $\xi_i\leq 2mL$,
most of the terms in the factorization are smaller than $1$:
\begin{multline}
 Z_{[x,\cT_1],h}^{\go,\gb} \leq Z_{[x,\tT_{a+1}],h}^{\go,\gb} \prod_{i=a+1}^{J_1} \max_{y\geq 0}
  \left[ \sum_{t=1}^{\xi_i} \frac{K(t+y)}{\bar K(y)} Z_{[\tT_{i-1}+t,\tT_i],h}^{\go,\gb}+ \frac{\bar K(\xi_n+y)}{\bar K(y)}\right]
\\
  \leq Z_{[x,\tT_{a+1}],h}^{\go,\gb}.
\end{multline} 

Then from this and equation \eqref{bornpamal} in Lemma \ref{superlem1}, one has
\begin{equation}
  \sum_{x\in(\tT_{a},\tT_{a+1}]} K(x+y)  Z_{[x,\cT_1],h}\le  \sum_{x=1}^{\xi_{a+1}} K(x+y+\tT_a) Z_{[\tT_a+x,\tT_{a+1}],h} \le 
\sum_{x\in(\tT_{a},\tT_{a+1}]} K(x+y),
\end{equation}
which ends the proof of \eqref{facil}.

%
%

Let us deal with the case $x\in (0,\xi_1]$. One needs a statement analogue to the one of Lemma \ref{parigolo}, that is

\begin{lemma}\label{patrorigolo}
 There exists a constant $C_{26}<1$ such that for any $x\in (0,\xi_1]$,
\begin{equation}
 Z_{[x,\cT_1],h}^{\go,\gb} \leq e^{C_{18} h^{1/\ga} \xi_1} \prod_{b=2}^{J_1} \left( 1- C_{26} \frac{\xi_b}{ \tT_{b-1}}\right).
\label{bound1}
\end{equation} 
\end{lemma}
Note that the second line of \eqref{pludur} is an immediate consequence of this Lemma.

\begin{proof}[Proof of Lemma \ref{patrorigolo}]
 One uses the coarse graining procedure similar to the one of Lemma \ref{lem2} to get
\begin{equation}
 Z_{[x,\cT_1],h}^{\go,\gb} \leq Z_{\xi_1,h} \prod_{b=2}^{J_1}
    \max_{y\in [0,\tT_{b-1}]}  \left[ \sum_{t=1}^{\xi_b} \frac{K(t+y)}{\bar K(y)} Z_{[\tT_{b-1}+t,\tT_b],h}^{\go,\gb}+
 \frac{\bar K(\xi_b+y)}{\bar K(y)}\right],
\end{equation} 
One uses  equation \eqref{borntobealive} to bound  $Z_{\xi_1,h}$.
As for the other factors of the product, one already has good bounds on them 
thanks to Lemma \ref{superlem}. Indeed, equation \eqref{bornpamal} gives directly
\begin{equation}
 \sum_{t=1}^{\xi_b} \frac{K(t+y)}{\bar K(y)} Z_{[\tT_{b-1}+t,\tT_b],h}^{\go,\gb}+ \frac{\bar K(\xi_b+y)}{\bar K(y)}
\le (1- C_{17}) \sum_{t=1}^{\xi_b} \frac{K(t+y)}{\bar K(y)}\le 1-C_{26} \frac{\xi_b}{ \tT_{b-1}},
\end{equation}
where the last inequality holds for all $y\in [0,\tT_{b-1}]$ and is
obtained in the same way that \eqref{estimmin1}.
\end{proof}
%
%

%
%

We are now ready to prove \eqref{pludur}. If
 $\xi_1 \leq L\log L$ and $J_1 \geq L^{(\tga-1)/2}$, then from Lemma~\ref{patrorigolo},
\begin{equation}
 \log Z_{[x,\cT_1],h}^{\go,\gb} \leq C_{18} h^{1/\ga} L \log L - C_{26} \sum_{b=2}^{J_1} \frac{\xi_b}{\tT_{b-1}},
\label{finlemtosto2}
\end{equation} 
where we used that $\xi_b/\tT_{b-1}\leq \xi_b/\xi_1\leq 1$, and $C_{26}<1$.
Moreover,
one also has 
\begin{equation}
 \sum_{b=2}^{J_1} \frac{\xi_b}{\tT_{b-1}} \geq \frac12 \log \left( \cT_1 /\xi_1 \right)
\end{equation}
 (see \eqref{sumxi}), so that with our assumptions
$\cT_1\geq J_1\geq L^{(\tilde\ga+1)/2}$ and $\xi_1\leq L\log L$, 
the inequality \eqref{finlemtosto2} gives (recall also that $L=\lfloor C_{16}h^{-1/\ga} \rfloor$)
\begin{equation}
 \log Z_{[x,\cT_1],h}^{\go,\gb} \leq C_{18}C_{16} \log L - \frac{C_{26}}{2}\log \left( L^{(\tga-1)/2}/\log L \right),
\end{equation} 
which is negative if one chooses $C_{16}$ small enough, and $h$ sufficiently small (so that $L(h)$ is large).
\end{proof}

\section{Proof of Theorem \ref{ncontacts}}
\label{sec:contact}

As for Theorem \ref{freeE}, the cases $\alpha<1$ and $\alpha>1$ present some dissimilarities and therefore
the details for them will be treated separately.
However, in the first part of this section, we give the ideas behind the proof and its first step for the two cases.
As we always have in this section $h=0$, we drop dependence in $h$ in the notation.

Recall the definition \eqref{defenvir} of our environment $\go$. 
For any event $A$, define 
\begin{equation}
 Z^{\go,\gb}_{N}(A):= \bE\left[e^{\sum_{n=1}^N \gb\go_n\ind_{\{n\in \tau\}}}\ind_{\{\tau \in A\}}\right].
\end{equation}
%

We prove Theorem \ref{ncontacts} (in fact a finer result that gives an estimate on the asymptotic of the tail behavior of $|\tau\cap [0,N]|$).

\begin{proposition}\label{ncontactvrai}
For almost every $\go$, for every $\gep>0$ there exists some $a_0$  (depending on $\go$, $\gb$ and $\gep$) and some $\gd=\gd(\gep)$ 
which can be made arbitrarily small, such that for all $a\ge a_0$ for all $N\in\N$ one has:\\
if $\alpha>1$
\begin{equation}\begin{split}\label{akak}
  Z^{\go,\gb}_{N}(|\tau\cap [0,N]|=a)&\le a^{\gep}N^{-\alpha}\max(a^{-\tilde\alpha\alpha},N^{-1}), 
\quad \text{ if }  a \le N^{\frac{1}{\tilde\alpha}+\gd},\\
   Z^{\go,\gb}_{N}(|\tau\cap [0,N]|=a)&\le e^{-N^{\gd^4}} \quad \text{ if }  a > N^{\frac{1}{\tilde\alpha}+\gd};
\end{split}
\end{equation}
and if $\alpha<1$
\begin{equation}\begin{split}\label{akak2}
  Z^{\go,\gb}_{N}(|\tau\cap [0,N]|=a)&\le a^{\gep}N^{-\alpha} a^{-\tilde\alpha}, 
\quad \text{ if }  a \le N^{\frac{\alpha}{\tilde\alpha}+\gd},\\
 Z^{\go,\gb}_{N}(|\tau\cap [0,N]|=a)&\le e^{-N^{\gd^4}} \quad \text{ if }  a > N^{\frac{\alpha}{\tilde\alpha}+\gd}.
\end{split}
\end{equation}

\end{proposition}

From the above proposition, that we prove in Section \ref{sec:ncontactvrai},
we get the following result that is stronger than Theorem~\ref{ncontacts}, and gives
an upper tail for the number of contacts points.

\begin{cor}
\label{cor:ncontactvrai}
For almost every $\go$, for every $\gep$, there exist some $\gd>0$ and a constant $\mathrm{C}=\mathrm{C}(\go,\gb,\gep)$ such that for all $N$,
 for every $a\le N^{\frac{1\wedge\ga}{\tilde\ga}+\gd}$
\begin{equation}
\bP_N^{\go,\gb} \left(|\tau\cap[0,N]|=a\right)\le 
\begin{cases}
 \mathrm{C} a^{\gep-\tilde\alpha} & \quad\text{if } \ga<1, \\
 \mathrm{C} a^{\gep} \max(a^{-\tilde\alpha\alpha},N^{-1}) & \quad\text{if } \ga>1,
\end{cases}
\end{equation}
and 
\begin{equation}
\bP_N^{\go,\gb}\left(|\tau\cap[0,N]|\ge N^{\frac{1\wedge\ga}{\tilde\ga}+\gd} \right)\le \mathrm{C} e^{-N^{\gd^4/2}}.
\end{equation}
Moreover
\begin{equation}
 Z^{\go,\gb}_{N}\le \mathrm{C}N^{-\alpha}.
\end{equation}
\end{cor}

\begin{proof}
We prove everything in the case $\alpha<1$ the other case being similar. Let us start with the last statement. 
Fix $\gep>0$ small, and then some $\gd\le \gep$ and $a_0$ such that Proposition~\ref{ncontactvrai} holds for $\gep$. Then,
$a_0$ being fixed,
there exist a constant $C(a_0)$ such that for all $a\leq a_0$
\begin{equation}\label{wouhou}
 Z^{\go,\gb}_{N}(|\tau\cap [0,N]|= a)\le  \bP(|\tau\cap [0,N]|=a)\le a C(a_0) N^{-\alpha},
\end{equation}
where we used Proposition \ref{prop:doney} to get the last inequality.
This, together with the estimates \eqref{akak2}, implies that 
\begin{multline}
 Z^{\go,\gb}_{N}= Z^{\go,\gb}_{N}(|\tau\cap [0,N]|< a_0)+\sum_{a= a_0}^{\infty} Z^{\go,\gb}_{N}(|\tau\cap [0,N]|= a)\\
  \leq C(a_0) N^{-\ga}  \sum_{a=0}^{a_0-1} a+ N^{-\ga}\sum_{a= a_0}^{\infty} a^{\gep-\tga}\le \mathrm{C} N^{-\alpha}.
\end{multline}
For the first two statements, one uses that 
\begin{equation}
 Z^{\go,\gb}_{N}\ge \bP(\tau_1>N)\ge C_{27} N^{-\alpha},
\label{lowerZ}
\end{equation}
for some constant $C_{27}>0$. Combined with \eqref{akak2} (or with \eqref{wouhou} for $a<a_0$),
this gives the right bound for the first statement for $a\leq N^{\frac{\ga}{\tilde\ga}+\gd}$.
The second statement is also an easy consequence of \eqref{lowerZ} and \eqref{akak2}, writing
\begin{multline}
\bP_N^{\go,\gb}\left(|\tau\cap[0,N]|\ge N^{\frac{\ga}{\tilde\ga}+\gd} \right)\le
     \frac{1}{Z^{\go,\gb}_{N}} Z_{N}^{\go,\gb}\left(|\tau\cap[0,N]|\ge N^{\frac{\ga}{\tilde\ga}+\gd} \right)\\
   \leq (C_{27})^{-1} N^{\alpha} \sum_{k=N^{\frac{\ga}{\tilde\ga}+\gd} }^{N} e^{-N^{\gd^4}} \leq \mathrm{C} e^{-N^{\gd^4/2}}.
\end{multline}
\end{proof}

At the end of the Section, we prove the following result that complements
the above and gives a lower tail for the number of contact points
under $\bP_{N,h=0}^{\go,\gb}$.

\begin{proposition}\label{lowertail}
For almost every $\go$, for any $\gep$ there exists $a_0$ such that
for $a\geq a_0$, and $a\leq N^{\frac{1\wedge\ga}{\tilde\alpha}-\gep}$ one has
\begin{equation}
  Z^{\go,\gb}_{N}(|\tau\cap [0,N]|=a)\ge a^{-\gep} N^{-\alpha}a^{-\frac{\tilde \alpha(\alpha+1)-1}{1 \wedge \alpha}}.
\label{contactslowbound}
\end{equation}
and
\begin{equation}
  \bP^{\go,\gb}_{N}(|\tau\cap [0,N]|=a)\ge a^{-\gep-\frac{\tilde \alpha(\alpha+1)-1}{1 \wedge \alpha}}.
\end{equation}
\end{proposition}

Note that Corollary \ref{cor:ncontactvrai} and Proposition \ref{lowertail} give
respectively the upper and the lower bound in Proposition \ref{prop:tail}.

%

We recall briefly here Section \ref{sec:strategy} which describes the
strategy to adopt to prove Proposition \ref{ncontactvrai}. Recall the definition
\eqref{defVtau} of $V_N^{\hat\tau}(\tau)$, the number of $\hat\tau$-stretches
visited by $\tau$, and inequality \eqref{pioutrrr}
\begin{equation}\tag{\ref{pioutrrr}}
 \bbE\left[e^{\sum_{n=1}^N \gb\go_n\ind_{\{n\in \tau\}}}\right]\le \left(\frac{1+e^{-\gb}}{2}\right)^{V^{\hat \tau}_N(\tau)},
\end{equation}
 where $\bbE$ denotes the average only on the values of 
 $\{X_i\}_{i\in\N}$, i.e. on the disorder $\omega$ conditionally on the
 realization of $\hat \tau$.
One estimates in Lemmas \ref{durdur} and \ref{durdur2} the contribution of trajectories of $\tau$ that visit many
$\hat\tau$-stretches,
and in Lemmas \ref{pasmieu} and \ref{pasmieu2} the contribution of trajectories of $\tau$ that visit few
$\hat\tau$-stretches.

\subsection{Proof of Proposition \ref{ncontactvrai} in the case $\alpha>1$}
\label{sec:ncontactvrai}

We prove the Proposition from the two following Lemmas.

\begin{lemma}\label{durdur}
Given $\gd>0$, there exists some $x_0(\go,\gb,\gep)$ such that for every $x\ge x_0$, and every 
$y\in[x,x^{\tilde \alpha(1-\gd)}]$  
one has
\begin{equation}
 Z^{\go,\gb}_{y}(\tau_x=y)\le e^{-x^{\gd/2}}.
\end{equation}
\end{lemma}

 \begin{lemma}\label{pasmieu}
If $\alpha>1$,
for any $\gep>0$ there exists some $\gd>0$ and $a_0\in \N$ such that for all $a\ge a_0$, and for all $N\in \N$ one has
\begin{equation}
\bP\left[ |\tau\cap[0,N]|= a\ ; \ \forall x\in [a^{\gd},a-1], \tau_x > x^{\tilde \alpha(1-\gd)}\right]
\le a^{\gep}N^{-\alpha} \max(a^{-\alpha\tilde\alpha},N^{-1})/2.
 \end{equation}
\end{lemma}


\begin{proof}[Proof of Proposition \ref{ncontactvrai}]
Let us fix $\gep>0$. As $\go$ is non-positive, the definition of $Z^{\go,\gb}_N(A)$ implies that for every $A$
\begin{equation}
Z^{\go,\gb}_N(A)\le \bP(A).
\end{equation}
Therefore, Lemma \ref{pasmieu} gives us directly that one can find $\gd$ such that for $a$ large enough one has
\begin{equation}
Z^{\go,\gb}_{N}(|\tau\cap[0,N]|= a\ ; \ \forall x\in [a^{\gd},a-1], \tau_x > x^{\tilde \alpha(1-\gd)})
\le a^{\gep}N^{-\alpha}\max(a^{-\alpha\tilde\alpha},N^{-1})/2.\label{truc0}
\end{equation}
Let us show now that
\begin{equation}\label{truc}
 Z^{\go,\gb}_{N}(|\tau\cap[0,N]|= a,  \ \exists x\in [a^{\gd},a-1], \tau_x 
\le x^{\tilde \alpha(1-\gd)})\le a^{\gep}N^{-\alpha}\max(a^{-\alpha\tilde\alpha},N^{-1})/2,
\end{equation} 
(which combined with \eqref{truc0} gives the first part of \eqref{akak}).
We do so by decomposing over all possible values for $x$ and $\tau_x$.
\begin{multline}
  Z^{\go,\gb}_{N}(|\tau\cap[0,N]|= a,  \ \exists x\in[ a^{\gd},a-1], \tau_x 
\le x^{\tilde \alpha(1-\gd)})\le \sum_{x=a^{\gd}}^{a-1} \sum_{y=x}^{x^{\tilde \alpha(1-\gd)}} Z^{\go,\gb}_{N}(\tau_x=y ; \tau_a>N)\\
\le \sum_{x=a^{\gd}}^{a-1} \sum_{y=x}^{x^{\tilde \alpha(1-\gd)}\wedge N} Z^{\go,\gb}_{y}(\tau_x=y)\bP(\tau_{a-x}>N-y).
\end{multline}
Using Lemma \ref{durdur} one gets that the above is smaller than

\begin{equation}\label{trucre}
 \sum_{x=a^{\gd}}^{a-1} \sum_{y=x}^{x^{\tilde \alpha(1-\gd)}\wedge N} e^{-x^{\delta/2}} \bP(\tau_{a-x}>N-y).
\end{equation}
If $a \ge N^{\delta}$ then $e^{-x^{\delta/2}}\le e^{-N^{\gd^3/2}}$ so that \eqref{trucre} is smaller than
$N^2 e^{-N^{\gd^3/2}}$ and \eqref{truc} holds.
If $a\le N^{\delta}$ and $\delta$ is small enough, from Proposition \ref{doneybis}, $\bP(\tau_{a-x}>N-y)\le 2 a\bar K(N)$, 
and therefore one has
\begin{equation}
 \sum_{x=a^{\gd}}^{a-1} \sum_{y=x}^{x^{\tilde \alpha(1-\gd)}\wedge N} e^{-x^{\delta/2}} \bP(\tau_{a-x}>N-y)
 \le 	2a a^{\tilde \alpha(1-\gd)+1}e^{-a^{\delta^2/2}}\bar K(N)
\end{equation}
which implies \eqref{truc}. 

\medskip

For the case $a> N^{\frac{1}{\tilde\alpha}+\gd}$, the left-hand side of \eqref{truc0} is equal to zero for $\gd$ small enough,
since the condition $\tau_a >a^{\tilde\ga (1-\gd)}>N^{1+\gd(\tga-1-\gd)}$ would contradict the event $\{|\tau\cap[0,N]|=a\}$.
Moreover
the left-hand side of \eqref{truc} is smaller than $N^2 e^{-N^{\gd^3/2}}\le e^{-N^{\gd^4}}$ for $N$ large enough so that  
Proposition \ref{ncontactvrai} is proved.
\end{proof}

\begin{proof}[Proof of Lemma \ref{durdur}]
Note that if one wants to visit only a few $\hat\tau$-stretches, one has to put
a lot of contacts in very few $\hat\tau$-stretches.
One then notices that according to Lemma \ref{lem:boundmax}, if $y$ is larger than some $N_{0}(\hat \tau)$,
the longest $\hat \tau$-stretch in the interval $[0,y]$ is of length smaller 
than $y^{1/\tilde \alpha}\log y\le \tilde \alpha x^{1-\delta}\log x$ for
 any $y\le x^{(1-\delta)\tilde \alpha}$. For that reason if $\tau_x=y$, with $x\geq N_0(\hat\tau)$ and for the
values of $y$ considered, there cannot be a $\hat\tau$-stretch longer than $x^{1-3\gd/4}$ so that
\begin{equation}
V^{\hat \tau}_x(\tau)\ge \frac{x}{\max\{\tau_{i+1}-\tau_i\ | \ \tau_i\le y\}}\ge  x^{3\gd/4},
\end{equation}
and from \eqref{pioutrrr} one gets that for $x$ large enough
\begin{equation}
 \bbE\left[ \sum_{y=x}^{x^{(1-\delta)\tilde \alpha}} Z^{\go,\gb}_y (\tau_x=y)\right]
     \leq \sum_{y=x}^{x^{(1-\gd)\tga}} \bP(\tau_x=y)\left(\frac{1+e^{-\gb}}{2}\right)^{x^{3\gd/4}} 
     \le \left(\frac{1+e^{-\gb}}{2}\right)^{x^{3\gd/4}}
\end{equation}
Using the Markov inequality and the Borel-Cantelli Lemma, one gets that there exists a (random) $x_0(\go)$ such that for all $x\ge x_0(\go)$
\begin{equation}
\sum_{y=x}^{x^{(1-\delta)\tilde \alpha}} Z^{\gb,\go}_y(\tau_x=y)\le \exp\left(-x^{\gd/2}\right).
\end{equation}
\end{proof}

The condition $\forall x\ge a^{\gd}, \tau_x> x^{(1-\gd)\tilde \alpha}$ implies that $\tau_x$ is
stretched out at all scales, and one has to sum over
the different ways of stretching $\tau$.
Thus Lemma \ref{pasmieu} requires a multi-scale analysis and for the sake of clarity, we restate
it in an apparently more complicated version.
One reason for doing so is that it allows to do a proof by induction.

\begin{lemma}
For all values of $l\in\N$, if $\gd_2\le \gd(l)$ there exists a constant $C(l)$ such
that for all $N\in\N$ and $a\in\N$ large enough with $a\le N^{\frac{1}{\tilde \alpha-\gd_2}}$, one has
\begin{multline}\label{albur}
  \max_{d\in [0,a^{\alpha^{-l}(\tilde \alpha-\gd_2)^{-l+1}}/2]} \bP \left[ |\tau\cap[0,N-d]|=a\ ; \ 
\forall x \in[ a^{\gd_2},a], \tau_x> x^{\tilde \alpha-\gd_2}-d\right]
\\
\le C(l) N^{-\alpha} a^{(\alpha(\tilde \alpha-\gd_2))^{-l}}
\max(a^{-\alpha(\tilde \alpha-\gd_2)},N^{-1}).
 \end{multline}
\label{lem:multiscale1}
\end{lemma}

\begin{rem}\rm
\label{rem:multiscale}
The probability of the event on the right hand side of \eqref{albur} is zero when $a> N^{\frac{1}{\tilde \alpha-\gd_2}}$ as 
 $|\tau\cap[0,N-d]|=a$ implies $\tau_{a-1}\le N-d$. Therefore the result holds in fact for all $a$.
Using \eqref{wouhou} one notices that the result holds for all $a$ and $N$ (after eventually changing the constant $C(l)$).

\smallskip

One gets Lemma \ref{pasmieu} from this by taking $\gd_2$ small enough and $l$, $a$ large enough and $d=0$.
The reason we prove the result for all $d\in [0,a^{\alpha^{-l}(\tilde \alpha-\gd_2)^{-l+1}}/2]$ and not only for $d=0$
 is to make the induction step in the proof work.
\end{rem}

\begin{proof}

We introduce some additional notation that will make the proof more readable.
We define for all $j\geq 0$
\begin{equation}\begin{split}
x_j&:=a^{(\alpha(\tilde \alpha-\gd_2))^{-j}} \quad (x_0=a),\\
y_j&:=\frac12 x_j^{\tilde\ga-\gd_2}=\frac12 a^{\ga^{-j}(\tilde\ga-\gd_2)^{-j+1}}.\\
\end{split}\end{equation} 
With these notation \eqref{albur} reads
\begin{multline}\label{inductionst}
  \max_{d\in [0,y_l]} \bP \left[ |\tau\cap[0,N-d]|=a\ ; \
\forall x\in[ a^{\gd_2},a-1], \tau_x\ge x^{\tilde \alpha-\gd_2}-d\right]\\
\le C(l) N^{-\alpha}x_l \max(y_0^{-\ga},N^{-1}).
\end{multline}
Note that $x_j$ and $y_j$ are decreasing in $j$, and also that 
provided that $\gd_2$ is small enough one has for any $j\ge 0$, that both $x_j$ and $y_j$ tends to infinity with $a$ and
\begin{equation}\label{lunplugran}
y_j\gg x_j.
\end{equation}
Let us start with the proof of the case $l=0$. 
On the event we consider, $\tau_{a-1}$ has to be larger
than $(a-1)^{\tilde \alpha-\gd_2}-d$, i.e.\ larger than what it would typically be under $\bP$.
We use Proposition \ref{doneybis} to bound from above the probability of this event.
The quantity we have to bound is smaller than
\begin{multline}
 \bP\left[\tau_a>N-d \ ;\  \tau_{a-1} \in \big((a-1)^{\tilde \alpha-\gd_2}-d,N-d\big]\right]\\
=
\sum_{y=(a-1)^{\tilde \alpha-\gd_2}+1-d}^{N-d} \bP(\tau_{a-1}=y)\bP(\tau_1> N-y-d)\\
=(1+o(1))\!\!\!\! \sum_{y=(a-1)^{\tilde \alpha-\gd_2}+1-d}^{N-d} \!\!\!\! a K(y)\bar K(N-d-y)\\
\le
   C(0) a N^{-\alpha} (a^{-\alpha(\tilde\alpha-\gd_2)}\vee N^{-1}),
\end{multline}
where here (and later in the proof) $o(1)$ denotes a quantity that goes to zero
when both $a$ and $N$ are large. Proposition \ref{doneybis} was used to get from the second to the third line,
the last inequality coming from a straightforward computation, using the assumption on $K(\cdot)$.

\medskip

 We assume now that \eqref{inductionst} holds  for all $l'< l$ and prove it for $l$. 
 Fix $d\le y_l$. Assume that $\gd_2=\gd_2(l)$ is small enough, so that $x_l\ge a^{\gd_2}$.
We decompose over all the possible values for $\tau_{x_l}$ and use the Markov property for the renewal process.
The l.h.s.\ of \eqref{inductionst} is smaller than
\begin{multline}\label{ddone}
 \bP\left[|\tau\cap[0,N-d]|=a\ ;\ \forall x\in[ x_l,a-1], \tau_x> x^{\tilde \alpha-\gd_2}-d\right]
=
\sum_{d_1=2y_l+1}^{N} \bP\left(\tau_{x_l}=d_1-d\right)\\ \times
\bP\left[|\tau\cap[0,N-d_1]|=a-x_l \ ; \forall x \in[0,a-x_l-1] ,
\tau_x\ge  (x+x_l)^{\tilde \alpha-\gd_2}-d_1\right].
\end{multline}
 On the event we are considering in $\eqref{inductionst}$, $\tau_{x_l}$ has to be larger
than $2y_l-d\ge y_l$ i.e. larger than what it would typically be under $\bP$ (cf. \eqref{lunplugran}).
Therefore $\bP\left(\tau_{x_l}=d_1-d\right)$ can always be estimated by using Proposition \ref{doneybis}.
If $d_1\le y_i$, the quantity
\begin{multline}
\bP\left[|\tau\cap[0,N-d_1]|=a-x_l \ ; \forall x \in[0,a-x_l] ,
\tau_x\ge  (x+x_l)^{\tilde \alpha-\gd_2}-d_1\right]\\
\le \bP\left[|\tau\cap[0,N-d_1]|=a-x_l \ ; \forall x \in[0,a-x_l] ,
\tau_x\ge  x^{\tilde \alpha-\gd_2}-d_1\right]
\end{multline}
can be estimated by using the induction hypothesis \eqref{inductionst} for $i<l$. 
For this reason we decompose the sum in the right hand side of \eqref{ddone} in $l$ terms,
corresponding to $d_1\in (2y_l,y_{l-1}]$, $d_1\in(y_j,y_{j-1}]$ ($j\in\{1,\dots l-1\}$)  and $d_1\in (y_0, N]$.
When $d_1>y_0$ one cannot use the induction step and for this reason  the contribution from  $d_1\in (y_0, N]$ is dealt with separately.

Notice that
\begin{multline}\label{perzr}
 \sum_{d_1 \in (y_j,y_{j-1}]}  \bP\left(\tau_{x_l}=d_1-d\right) \\
 \times \bP\left[|\tau\cap[0,N-d_1]|=a-x_l \ ; \forall x \ge 0 ,
\tau_x\ge  (x+x_l)^{\tilde \alpha-\gd_2}-d_1\right]\\
\le (1+o(1))\sum_{d_1 \in (y_j,y_{j-1}]} x_l K(d_1-d) C(j-1)  N^{-\alpha}x_{j-1}\max(y_0^{-\ga},N^{-1}) \\ 
\le C'(j)x_l y_j^{-\alpha}   x_{j-1}\max(y_0^{-\ga},N^{-1}).
\end{multline}
From the definitions of $x_j$ and $y_j$ one has
$y_j^{\alpha}=\frac{1}{2^{\ga}} x_{j-1}$, so that $y_j^{-\ga}x_{j-1}= 2^{\alpha}$ for all $j\geq1$.
The term corresponding to $d_1\in (2y_l,y_{l-1}]$ can be dealt with in the same manner.

\smallskip

Now we estimate the sum on $d_1\in (y_0, N]$. By Proposition \ref{doneybis} one has
\begin{multline} \label{lastpiece}
 \sum_{d_1 \in  (y_0, N]}  \bP\left(\tau_{x_l}=d_1-d\right)\bP\left( |\tau \cap [0,N-d_1]|=a-x_l\right)\\
\le (1+o(1))\sum_{d_1 \in  (y_0, N]} x_l K(d_1-d)\bP\left( |\tau \cap [N-d_1]|=a-x_l\right).
\end{multline}
If $d_1$ is less that $N/2$, then choosing $\delta$ small enough 
\begin{multline} \label{stroumfphlala}
\bP\left( |\tau \cap [N-d_1]|=a-x_l\right)=
 \sum_{x=1}^{N-d_1} \!\!\! \bP\left(\tau_{a-x_l-1}=x\right) \bar K(N-d_1-x)\\
=\sum_{x=1}^{N^{1-\delta}} \bP\left(\tau_{a-x_l-1}=x\right) \bar K(N-d_1-x)+ 
\sum_{x=N^{1-\delta}+1}^{N-d_1} \bP\left(\tau_{a-x_l-1}=x\right) \bar K(N-d_1-x)
\\=(1+o(1))\left[ \bar K(N-d_1)+\sum_{x=N^{1-\delta}+1}^{N-d_1} a K(x) \bar K(N-d_1-x)\right]
\le c(l) N^{-\alpha},
\end{multline}
where we made use of $\bar K(N-d_1-x)=\bar K(N-d_1)(1+o(1))$ for $x\le N^{1-\delta}$, and of Proposition \ref{doneybis}
for $x>N^{1-\gd}$. Note that we also used the restriction $a\leq N^{\frac{1}{\tga-\gd_2}}$ for the last inequality,
to get that $aN^{-\ga(1-\gd)}\leq 1$. Hence one has
\begin{equation} \label{lesd}
\sum_{d_1 \in  (y_0, N/2]} x_l K(d_1-d)\bP\left[ |\tau \cap [0,N-d_1]|=a-x_l\right]\le
 c'(l) x_l y_0^{-\alpha}N^{-\alpha},
\end{equation}
for $c'(l)$ large enough.
To estimate the contribution of $d_1\in(N/2,N]$, one notices that 
\begin{multline}
 \sum_{L=0}^{\infty}\bP\left( |\tau \cap [0,L]|=a-x_l\right)= \bE\Big[ \#\{L\in\R,\ L\in[\tau_{a-x_l},\tau_{a-x_l+1}) \} \Big]\\
   = \bE[\tau_{a-x_l+1}-\tau_{a-x_l}]=\bE[\tau_1].
\end{multline}
so that
\begin{multline}
\sum_{d_1 \in (N/2,N]} x_l K(d_1-d)\bP\left( |\tau \cap [N-d_1]|=a-x_l\right)\\
\le  c(l) x_l N^{-(1+\alpha)}\sum_{d_1 \in (N/2,N]} \bP\left[ |\tau \cap [N-d_1]|=a-x_l\right]\le c'(l)x_l N^{-(1+\alpha)}.
\end{multline}
This together with \eqref{perzr} and \eqref{lesd} gives the result.

\end{proof}

\subsection{Proof of Proposition \ref{ncontactvrai} in the case $\ga<1$}

One has to adapt Lemmata \ref{durdur} and \ref{pasmieu} to this new case.
The difference lies in the following fact: as here the renewal does not
have finite mean, one needs a stretch of length much longer than $x$ 
to set $x$ contacts on the defect line.

\begin{lemma}\label{durdur2}
Given $\gd>0$, there exists some $x_0(\go,\gb,\gep)$ such that for every $x\ge x_0$, every 
$y\in[x,x^{\frac{\tilde \alpha}{\alpha}(1-\gd)}]$  
one has
\begin{equation}
 Z^{\go,\gb}_{y}(\tau_x=y)\le e^{-x^{\gd/8}}.
\end{equation}
\end{lemma}

 \begin{lemma}\label{pasmieu2}
If $\alpha<1$,
for any $\gep>0$ there exists $\gd>0$ and $a_0\in \N$ such that for all $a\ge a_0$, for all $N$ one has
\begin{equation}
\bP\left[ |\tau\cap[0,N]|= a\ ; \ \forall x\in[ a^{\gd}a], \tau_x > x^{\frac{\tilde \alpha}{\alpha}(1-\gd)}\right]
\le a^{\gep-\tilde\alpha}N^{-\alpha}/2.
 \end{equation}
\end{lemma}

The proof from the two Lemmas of the case $\ga<1$ in Proposition \ref{ncontactvrai}
is exactly the same as in the case $\alpha>1$, and therefore we leave it to the reader.

\begin{proof}[Proof of Lemma \ref{durdur2}]
 First note that if one wants to visit only a limited number of stretches after $x$ jumps (say less than $x^{\gd/2}$),
one must do at least $x^{1-\gd/2}$ jumps in the same stretch.
On the other hand, note that provided $x$ is large enough, from Lemma \ref{lem:boundmax} 
the longuest $\hat\tau$-stretch in $[0,y]$ for $y\leq x^{\frac{\tga}{\ga}(1-\gd)}$
has length smaller than
$x^{\frac{1-(3\delta/4)}{\alpha}}$.
For these reasons if $x$ is large enough, and
for the values of $y$ that we consider 
\begin{equation}
 \{ V_y^{\hat \tau}(\tau)\le x^{\gd/2} ; \tau_x=y\}\subset
\left\{ \exists t\in[0,x), (\tau_{t+x^{1-\gd/2}}-\tau_t)\le x^{\frac{1-(3\delta/4)}{\alpha}}\right\}.
\end{equation}
As a consequence

\begin{multline}\label{hyui}
 Z_y^{\go,\gb}( V_y^{\hat \tau}(\tau)\le x^{\gd/2}\ ; \ \tau_x=y )\le \bP\left(\exists t\in[0,x), 
(\tau_{t+x^{1-\gd/2}}-\tau_t)\le x^{\frac{1-(3\delta/4)}{\alpha}}\right)\\
\le x \bP\left(\tau_{x^{1-\gd/2}}\le x^{\frac{1-(3\delta/4)}{\alpha}}\right)
\le x \left[\bP\left(\tau_1\le x^{\frac{1-(3\delta/4)}{\alpha}}\right)\right]^{x^{1-\gd/2}}\le \frac{1}{2}e^{-x^{\gd/8}}
\end{multline}
if $x$ is large enough.
On the other hand according to \eqref{pioutrrr}
\begin{equation}
 \bbE\left[ \sum_{y=x}^{x^{\frac{\tilde \alpha}{\alpha}(1-\delta)}} Z_y^{\go,\gb}( V_y^{\hat \tau}{(\tau)}> x^{\gd/2}\ ; \tau_x=y)\right]
\le \left(\frac{1+e^{-\gb}}{2} \right)^{x^{-\delta/2}}.
\end{equation}
Using the Markov inequality and the Borel-Cantelli Lemma, one gets that there exists a (random) integer $x_0$ such that
for all $x\geq x_0$
\begin{equation}
 \sum_{y=x}^{x^{\frac{\tilde \alpha}{\alpha}(1-\delta)}} 
Z_y^{\go,\gb}\left( V_y^{\hat \tau}(\tau)> x^{\gd/2}\ ; \tau_x=y\right)\le e^{x^{-\delta/8}}/2,
\end{equation}
which together with \eqref{hyui} ends the proof.

\end{proof}

For Lemma \ref{pasmieu2} one proceeds as for Lemma \ref{pasmieu} and prove a recursive statement.

\begin{lemma}
For all values of $l$, if $\gd_2\le \gd(l)$ there exists a constant $C(l)$ such that for all $N\in\N$ and $a\in\N$ large enough 
with $a\le  N^{\ga(\tga-\ga\gd_2)^{-1}}$, one has
\begin{multline}\label{lustuc}
  \max_{d\in [0,a^{\alpha^{-1}(\tilde \alpha-\alpha \gd_2)^{-l+1}}/2]} \bP \left[ |\tau\cap[0,N-d]|=a\ ; \ 
\forall x\in[ a^{\gd_2},a-1], \tau_x> x^{\frac{\tilde \alpha}{\alpha}-\gd_2}-d\right]\\
\le C(l) N^{-\alpha}a^{(\tilde \alpha-\alpha\gd_2)^{-l}}a^{-(\tilde \alpha-\alpha\gd_2)}.
 \end{multline}
\end{lemma}
Note that Remark \ref{rem:multiscale} made for Lemma \ref{lem:multiscale1} applies also here.

\begin{proof}
This is very similar to the $\alpha>1$ case.
One uses some different notation this time:
\begin{equation}\begin{split}
x_j&:=a^{(\tilde \alpha-\alpha\gd_2)^{-j}} \quad (x_0=a),\\
y_j&:=\frac12 x_j^{\frac{\tilde\ga}{\alpha}-\gd_2}=\frac{1}{2} a^{\ga^{-1}(\tga-\ga\gd_2)^{-j+1}}.\\
\end{split}\end{equation} 
With these notation, \eqref{lustuc} reads
\begin{equation}\label{lustuc2}
  \max_{d\in [0,y_l]} \bP \left[ |\tau\cap[0,N-d]|=a\ ; \
\forall x\in [ a^{\gd_2},a-1], \tau_x\ge x^{\frac{\tilde \alpha}{\ga}-\gd_2}-d\right]
\le C(l) N^{-\alpha}x_l y_0^{-\ga}.  
\end{equation}
We also have that $x_j$ and $y_j$ are decreasing in $j$, and that
provided that $\gd_2$ is small enough, one has for any $j\ge 0$ that both $x_j$ and $y_j$ tends to infinity with $a$ and that
\begin{equation}\label{lunplugranii}
y_j\gg x_j^{\alpha}.
\end{equation}

We prove the statement first in the case $l=0$.
Note that on the event we consider, $\tau_{a-1}\gg a^{\alpha}$ i.e.\ $\tau_{a-1}$ has
to be much larger than what it would typically be under $\bP$.
Therefore one can use Proposition \ref{prop:doney} to estimate its probability. We get that the l.h.s.\ of \eqref{lustuc2} is smaller than 
\begin{multline}
  \bP \left[ \tau_a>N-d \ ; \  \tau_{a-1}\in ( (a-1)^{\frac{\tilde \alpha}{\alpha}-\gd_2}-d, N-d ]\right]\\
=\!\!\!\!\!\!\!\! \sum_{y=(a-1)^{\frac{\tilde \alpha}{\alpha}-\gd_2}+1-d}^{N-d}
    \!\!\!\!\!\!\!\bP[\tau_{a-1}=y] \bP\left[ \tau_1>N-d-y\right]
= (1+o(1)) \!\!\!\!\!\!\! \sum_{y=(a-1)^{\frac{\tilde \alpha}{\alpha}-\gd_2}+1-d}^{N-d} \!\!\!\!\!\!\! a K(y) \bar K (N-d-y)
\\ \le C(0) aN^{-\alpha}\max\left( a^{(\tilde \alpha-\alpha\gd_2)} , N^{-\alpha}\right)=C(0) a^{1-(\tilde \alpha-\alpha\gd_2)}N^{-\alpha}. 
\end{multline}
Proposition \ref{prop:doney} was used to get the third line.
The last equality comes from the fact that we consider only $a\le  N^{\ga\left(\tilde\alpha-\gd_2\ga\right)^{-1}}$.
Here (and later in the proof) $o(1)$ denotes a quantity that tends to zero when both $a$ and $N$ gets large.

%

\smallskip

We now assume the statement for all $l'< l$ and prove it for $l$. Fix $d\le y_l$.
Assume that $\gd_2=\gd_2(l)$ is small enough, so that $x_l\ge a^{\gd_2}$.
We decompose over all the possible values for $\tau_{x_l}$ and use the Markov property for the renewal process,
so that
the l.h.s.\ of \eqref{lustuc2} is smaller than
\begin{multline}\label{sumo}
 \sum_{d_1=2 y_l+1}^N \bP\left( \tau_{x_l}=d_1-d\right)\\
 \bP \left[ |\tau\cap[0,N-d_1]|=a-x_l\ ; \
\forall x\in[0,a-x_l-1], \tau_x >(x+x_l)^{\frac{\tilde \alpha}{\ga}-\gd_2}-d_1\right].
\end{multline}
Note that in the above sum,  one always has $\tau_{x_l}\ge 2y_l-d\ge y_l$, i.e.\  is much larger than the value it typically takes 
(cf. \eqref{lunplugranii})  under $\bP$. Therefore one can use Proposition \ref{prop:doney} to estimate the term $\bP\left( \tau_{x_l}=d_1-d\right)$.
As for the second term 

\begin{multline}
\bP \left[ |\tau\cap[0,N-d_1]|=a-x_l\ ; \
\forall x\in[0,a-x_l-1], \tau_x >(x+x_l)^{\frac{\tilde \alpha}{\ga}-\gd_2}-d_1\right]\\
\le
\bP \left[ |\tau\cap[0,N-d_1]|=a-x_l\ ; \
\forall x\in[0,a-x_l-1], \tau_x >x^{\frac{\tilde \alpha}{\ga}-\gd_2}-d_1\right],
\end{multline}
it can be bounded from above by using the induction hypothesis when $d_1\le y_i$, $i<l$.

For this reason we separate the contribution of the different terms 
 $d_1\in (2y_l,y_{l-1}]$, $d_1\in(y_j,y_{j-1}]$ ($j\in\{1,\dots l-1\}$)  and $d_1\in (y_0, N]$ in the sum \eqref{sumo}.
We just focus on the last one, as the computation for $d_1\le y_0$ is exactly the same as in Lemma
\ref{lem:multiscale1} (see \eqref{perzr}), using Proposition~\ref{prop:doney} instead of Proposition~\ref{doneybis}.
For  $d_1\in (y_0, N]$ one cannot use the induction hypothesis.
Using Proposition \ref{prop:doney} one gets
\begin{multline}
  \sum_{d_1\in (y_0,N]} \bP\left( \tau_{x_l}=d_1-d\right) \bP \left( |\tau\cap[0,N-d_1]|=a-x_l\right)\\
\leq(1+o(1))   \sum_{d_1\in (y_0,N]} x_l K(d_1)  \bP \left( |\tau\cap[0,N-d_1]|=a-x_l\right).
\end{multline}
As in \eqref{stroumfphlala} one shows that for $d_1\le N/2$, uniformly on the choice of $a\leq N^{\ga(\tga-\ga\gd_2)^{-1}}$ one has
\begin{equation}
 \bP \left( |\tau\cap[0,N-d_1]|=a-x_l\right)\le c(l) N^{-\alpha},
\end{equation}
so that
\begin{equation}
 \sum_{d_1\in (y_0,N/2]} x_l K(d_1)  \bP \left( |\tau\cap[0,N-d_1]|=a-x_l\right)\le c'(l)x_l y_0^{-\alpha} N^{-\alpha}.
\end{equation}
For the case $d_1>N/2$ one remarks that
\begin{multline}
  \sum_{L=0}^{N/2-1}\bP \left( |\tau\cap[0,L]|=a-x_l\right)= \bE\Big[ |\{L\in[0,N/2-1]\ : L \in [\tau_{a-x_l},\tau_{a-x_l+1})\}|\Big] \\
\le \bE\left[\max(\tau_1,N/2)\right].
\end{multline}
Therefore
\begin{multline}
 \sum_{d_1\in (N/2,N]} x_l K(d_1)  \bP \left[ |\tau\cap[0,N-d_1]|=a-x_l\right] \\
  \le c(l) x_l N^{-(1+\ga)} \bE\left[\max(\tau_1,N/2)\right] \leq c'(l)x_l N^{-2\ga}\le  c'(l)x_l (y_{0}N)^{-\ga} .
\end{multline}
The last inequality comes from the fact that $y_0\le N$ for the range of $a$ that we consider.

\end{proof}

\subsection{Proof of  Proposition \ref{lowertail} }

 Here the strategy consists in targeting directly the first $\hat\tau$-stretch with $\go\equiv0$, of
size larger than $2C_{28} a^{\frac{1}{1\wedge \ga}}$ (with $C_{28}$ a constant to be determined, depending only on $K(\cdot)$),
and then getting $a$ contacts in that stretch before exiting the system.
Define $i_{a}:=\min \{ i\ |\ \hat\tau_{i+1}-\hat\tau_{i} \geq 2C_{28} a^{\frac{1}{1\wedge \ga}} ,\ \go_{\hat\tau_{i+1}}=0\}$,
so that $\go \equiv 0$ on $(\hat\tau_{i_a},\hat\tau_{i_a+1}]$.
\smallskip

One wants to estimate  $i_a$ and $\hat\tau_{i_a}$. Let us define
\begin{equation}
\begin{split}
 M_N^*:= \max_{1\leq i\leq N} \{\hat\tau_{i+1}-\hat\tau_{i} \ |\ \go_{\hat\tau_i}=0\}.
\end{split}\end{equation}
Adapting the proof of Lemma \ref{lem:boundmax} one gets a random integer $N_0$ such that for all $N\geq N_0$

\begin{equation}
 M_N^*(\hat\tau) \geq   N^{1/\tga} (\log \log N)^{-1} .
\end{equation}
So that if $a$ is large enough
\begin{equation}
2C_{28} a^{\frac{1}{1\wedge \ga}}\ge  M_{i_a}^*\ge i_a^{1/\tga} (\log \log i_a)^{-1},
\end{equation}
and hence
\begin{equation}
i_a\le a^{\frac{\tga }{1\wedge \ga}}(\log a).
\end{equation}
By the law of large numbers for $\hat \tau$, the above inequality tranfers to $\hat \tau_{i_a}$: one also has
for $a$ large enough $\hat \tau_{i_a}\leq a^{\frac{\tga }{1\wedge \ga}}(\log a)$.
Note that under the assumption $a\leq N^{\frac{1\wedge\ga}{\tilde\ga}-\gep}$, one has $\hat \tau_{i_a}\ll N$.

\smallskip

Then, 
 decomposing $ Z^{\go,\gb}_{N}(|\tau\cap [0,N]|=a)$ according to the position of $\tau_1$ and $\tau_{a-1}$, and restricting to the event
$\tau_1\in(\hat\tau_{i_a},\hat\tau_{i_a}+C_{28} a^{\frac{1}{1\wedge \ga}}]$, one gets
\begin{multline}
 Z^{\go,\gb}_{N}(|\tau\cap [0,N]|=a) \geq \sum_{d=\hat\tau_{i_a}}^{\hat\tau_{i_a} + C_{28}a^{\frac{1}{1\wedge \ga}} }
        K(d) \sum_{f=d}^{\hat\tau_{i_a}+2C_{28}a^{\frac{1}{1\wedge \ga}}} \bP(\tau_{a-2}=f-d)\bP( \tau_1>N-f) \\
   \geq C_{29}a^{\frac{1}{1\wedge \ga}} \left( \hat \tau_{i_a}\vee a^{\frac{1}{1\wedge \ga}}\right)^{-(1+\ga)}
            \bP\left(\tau_{a-2} \leq C_{28}a^{\frac{1}{1\wedge \ga}}\right)  (N-\hat\tau_{i_a}-2C_{28}a^{\frac{1}{1\wedge \ga}})^{-\ga}
\label{boundinfcontacts}
\end{multline}
where we used the asymptotic properties of $K(\cdot)$ and the fact that $\hat \tau_{i_a}\ll N$.
Then one chooses the constant $C_{28}$ such that $\bP(\tau_{a-2} \leq C_{28} a^{\frac{1}{1\wedge \ga}})$
is bounded away from $0$ (take $C_{28}=2\bE[\tau_1]$ if $\ga>1$ and $C_{28}=1$ if $\ga<1$) and use our bound
on $\hat \tau_{i_a}$ to get the result.
\qed

{\bf Acknowledgements}: 
The authors are much indebted to G.\ Giacomin for having proposed to study such a model and for enlightening discussion about it,
as well as to F.L.\ Toninelli for his constant support in this project and his precious help on the manuscript. This work was 
initiated during the authors stay in the Mathematics Department of Università di Roma Tre, they gratefully acknowledge hospitality and support
H.L.\ acknowledges the support of ECR grant PTRELSS.

\begin{appendix}
\section{Renewal results}
We gather here a set of technical results concerning renewal processes.
They are used throughout the paper for the different renewals $\hat \tau, \tilde \tau$ and $\tau$, and 
therefore we state them for a generical renewal $\sigma=\{\gs_n\}_{n\geq0}$, starting from $\sigma_0=0$, whose law is denoted $\bP$, and
whose inter-arrival law satisfies
\begin{equation}
K(n):=\bP(\gs_1=n) = (1+o(1))c_{\gs} n^{-(1+\gz)},
\end{equation}
where $\xi>0$ and $\xi \ne 1$. 
We also assume that $\gs$ is recurrent, that is $K(\infty)=\bP(\gs_1=+\infty)=0$.
The results would stand still if $c_{\gs}$ was replaced by a slowly varying function but for the sake 
of simplicity, we restrict to the pure power-law case.
We have two subsections concerning respectively results for positive recurent renewals ($\gz>1$) and null-reccurent renewals ($\gz<1$).

\subsection{Case $\gz<1$}

We present a result of Doney concerning local-large deviation above the median for renewal processes. 

\begin{proposition}[\cite{Doney}, Theorem A]
If $\gz<1$, then one has that uniformly for $x\gg N^{\gz}$
\begin{equation}
 \bP\left(\gs_N = x \right) = (1+o(1)) N K\left(x\right).
\end{equation}  
More precisely, for any sequence $a_N$ such that $N^{\gz}=o(a_N)$ one has
\begin{equation}
 \lim_{N\to \infty} \sup_{x\ge a_N} \left| \frac{\bP(\sigma_N=x)}{N K(x)}-1 \right|=0.
\end{equation}

\label{prop:doney}
\end{proposition}



\subsection{Case $\gz>1$}

In this case we introduce $m=\bE[\tau_1]<\infty$.
We first prove the following equivalent of Proposition \ref{prop:doney}.
The proof present some similarities as well as some crucial differences with the one
in \cite{Doney}.
\begin{proposition}
For all $\gd>0$, one has uniformly for all $x\ge (m+\gd)N$.
\begin{equation}
 \bP\left(\gs_N = x \right)= (1+o(1)) N K\left(x-mN\right),
\end{equation}
or more precisely
\begin{equation}
 \lim_{N\to \infty} \sup_{x\ge (m+\gd)N} \left | \frac{\bP(\sigma_N=x)}{N K(x-mN)}-1\right|=0. 
\end{equation}
A simple consequence is that 
uniformly for $x\gg N$,
\begin{equation}
 \bP\left(\gs_N = x \right) = (1+o(1)) N K\left(x\right).
\end{equation}  
\label{doneybis}
\end{proposition}

\begin{rem}\rm
The idea behind this result (like for Proposition \ref{prop:doney}) is that if $\sigma_N$ has to be way above its median, 
the reasonable way to do it is to take all the excess in
one big jump, the rest of the trajectory being typical. Other strategies with several 
long jumps are proved to be comparatively unlikely. This is an important point to understand what is going on in
Sections \ref{prr},\ref{conrr} and \ref{sec:contact}.
\end{rem}
\begin{proof}
Given $\delta$, we set $\gep>0$ that is meant to be arbitrarily small. Take some $x\geq (m+\gd)N$.

\smallskip

Let us start with the lower bound, 
\begin{multline}
 \bP\left(\gs_N = x\right) \geq \bP\left(\gs_N=x\ ; \ \exists!\ i\in[1,N],\ \gs_{i}-\gs_{i-1}\ge \gep x\right) \\
  = N \sum_{y=\gep x}^{x} \bP(\gs_1=y) \bP\left( \gs_{N-1}=x-y\ ;\ \forall i\in [1,N-1],\ \sigma_i-\sigma_{i-1}\le \gep x \right) \\
   \ge N \!\!\! \min_{y\in [\gep x, x-(m-\gep)N]} \!\!\! K(y) \bP \big(\gs_{N-1} \in[(m-\gep)N, (1-\gep)x]\ ;\ \forall i\in [1,N-1],\
 \sigma_i-\sigma_{i-1}\le \gep x \big) .
\end{multline}
The second line is obtained by using independence and exchangability of the increments (decomposing over all $N$ possibilities for $i$),
and the third line by restricting to the values $y\in[\gep x, x-(m-\gep)N]$.
Then the assumption one has on $K(\cdot)$ guarantees that
\begin{equation}
 \min_{y\in[\gep x, x-(m-\gep)N]} K(y)=(1+o(1))K(x-(m-\gep)N).
\end{equation}
Using the law of large numbers for $\sigma_{N-1}$, one has that for $\gep$ sufficiently small
\begin{multline}
 \bP \left(\gs_{N-1} \in[(m-\gep)N, (1-\gep)x]\  ; \ \forall i\in [1,N-1],\
 \sigma_i-\sigma_{i-1}\le \gep x \right)\\
\ge   \bP \left(\gs_{N-1} \in[(m-\gep)N, (1-\gep)x]\right)- \bP\left(\exists i\in [1,N-1],\
 \sigma_i-\sigma_{i-1}\ge \gep x \right)\\
=1+o(1)+ NO((\gep x)^{-\xi} )=1+o(1).
\end{multline}
One gets the result by taking $\gep$ arbitrarily close to zero.

\medskip

For the upper bound it is easy to control the contribution of trajectories that make at least one large jump of order $x$.
We start with the more delicate part of controlling the contribution of trajectories that do not. We prove it to be negligible.
\begin{multline}
\bP\left(\gs_N=x\ ;\ \forall i\in[1,N],\ \gs_{i}-\gs_{i-1}\le  \gep x \right)\\
  \leq \bP\left(\gs_N=x\ ;\ \exists n_1,n_2\in[1,N]^2,\ \gs_{n_i}-\gs_{n_i-1}\in[ x^{1-\gep},\gep x] \text{ for } i=1,2 \right) \\ 
   + \bP\left(\gs_N=x\ ;\ \exists i\in[1,N],\ \gs_{i}-\gs_{i-1}\in[x^{1-\gep},\gep x]\ ;\ \forall j\neq i \ \gs_{j}-\gs_{j-1}< x^{1-\gep} \right)\\
  + \bP\left(\gs_N=x\ ;\ \forall j\in [1,N], \ \gs_{j}-\gs_{j-1}< x^{1-\gep} \right).
\label{2termes}
\end{multline} 
We can bound the first term by using the union bound on the different possibilities for $n_1$ and $n_2$ to get
some constant $C_{30}$
\begin{multline}
 \bP\left(\gs_N=x\ ;\ \exists i,j\in[1,N]^2,\ \gs_{i}-\gs_{i-1}\geq x^{1-\gep},\gs_{j}-\gs_{j-1}\geq  x^{1-\gep} \right) \\
   \leq \binom{N}{2} \sum_{y,z= x^{1-\gep}}^{  x} \bP\left( \gs_1=y \right) \bP\left( \gs_1=z \right)
          \bP\left(  \gs_{N-2}=x-y-z \right)
  \leq C_{30} N^2  x x^{-2(1-\gep)(1+\gz)} ,
\end{multline}
which smaller than $N x^{-2\zeta+2\gep(1+\gz)}$ uniformly in $x\ge N$. Hence this
term is negligible compared to the bound $N x^{-(1+\gz)}$  if 
$\gep$ is strictly smaller than $(1-\gz)/(2(1+\gz))$.

\smallskip
To estimate the other terms in \eqref{2termes}, define a renewal process $\bar\gs$ with $\bar\gs_0:=0$,
and $\bar \gs_i -\bar \gs_{i-1} :=(\gs_{i}-\gs_{i-1})\ind_{\{\gs_{i}-\gs_{i-1} < x^{1-\gep}\}}$.
One can bound the second and third term in the r.h.s.\ of \eqref{2termes} from above by
$ \bP\left( \bar \gs_{N-1}\geq (1-\gep)x \right)$.
Now we estimate this term by using Chernov bounds. For any positive $\gl$, one has
\begin{equation}
\bP\left( \bar\gs_N\ge (1-\gep)x \right)\le \bE\left[e^{\gl \bar \gs_1}\right]^N e^{-\gl (1-\gep) x}.
\label{borneexpodeb}
\end{equation}
Using the trivial bound $\bE[(\bar \gs_1)^k]\le (x^{1-\gep})^{k-1} m$, one finds that 
\begin{equation}
 \bE\left[e^{\gl \bar \gs_1}\right]\le 1+ \frac{m }{x^{1-\gep}}(e^{\gl x^{1-\gep}}-1). 
\end{equation}
If one chooses $\gl=o(x^{-1+\gep})$, one gets as $N$ goes to infinity
\begin{equation}
 \bE\left[e^{\gl \bar \gs_1}\right]^N \leq \exp\left( \gl m N (1+o(1)) \right),
\end{equation}
such that for $N$ large enough, 
\begin{equation}
\bP\left(  \bar\gs_N\ge (1-\gep)x \right)\leq \exp\left( \gl (mN-(1-\gep)x) (1+o(1))\right) \leq e^{-C_{31} x^{\gep/2}},
\label{borneexpofin}
\end{equation}
where the last inequality comes from taking $\gep$ small enough, 
and $\gl=x^{-1+\gep/2}$ (the constant $C_{31}$ depends only the choice of $\delta$). 
This is negligible compared to the bound one must obtain.

\smallskip

Then, we estimate the main contribution, using the union bound and exchangeability of the increments
\begin{multline}
 \bP\left(\gs_N=x\ ;\ \exists i\in[1,N],\ \gs_{i}-\gs_{i-1}\ge \gep x  \right) 
 \leq N \sum_{y=\gep x+1}^{x} \bP(\gs_1=y)\bP\left( \gs_{N-1} = x-y \right) \\
   \leq N\left[\max_{y\in [x-(m+\gep)N,x]} K(y) \bP\left( \gs_{N-1}\le (m+\gep)N \right)\right.\\
  \left. + \max_{y\in (\gep x, x-(m+\gep)N)}  K(y)  \bP\left( \gs_{N-1}> (m+\gep)N \right)\right].
\label{presquebon}
\end{multline}
The law of large numbers gives 
\begin{equation}
 \bP\left( \gs_{N-1}> (m+\gep)N \right)=o(1).
\end{equation}
On the other hand, one has from the assumption on $K(\cdot)$ that
\begin{equation}
 \begin{split}
  \max_{y\in [x-(m+\gep)N,x]} K(y)&=(1+o(1))K(x-(m+\gep)N),\\
  \max_{y\in [\gep x, x-(m+\gep)N]} K(y)&=(1+o(1))K(\gep x)=O(x^{-(1+\alpha)}).
 \end{split}
\end{equation}
This together with the fact that $\gep$ can be chosen arbitrarily close to zero gives the result.
%
\end{proof}

\smallskip
We finish with giving a result on the size of the longest inter-arrival interval
up to the $N^{\text{th}}$ jump,
\begin{equation}
 M_N:= \max_{1\leq i\leq N} \{\gs_i-\gs_{i-1}\}.
\label{defMn}
\end{equation}

\begin{lemma}
If $\gz>1$, there exists a random integer $N_0(\sigma)$ such that for all $N\ge N_0$
\begin{equation}
 N^{1/\gz} (\log \log N)^{-1}  \leq  M_N \leq N^{1/\gz}\log N.
\end{equation}
\label{lem:boundmax}
\end{lemma}

\begin{proof}
We use the fact that increments are i.i.d.\ to get
\begin{equation}
 \bP\left( M_N\leq A \right) =\bP(\gs_1\leq A)^{N}.
\end{equation}
Then, using that $\bP(\gs_1>A)$ is of order $A^{-\gz}$, one has
that there exist constants $C_{32},C_{33}>0$ such that
\begin{equation}
 \bP\left( M_{N}>  N^{1/\gz} \log N \right) \leq 1-\exp\left( -C_{32} (\log N)^{-\gz}\right) =(1+o(1))C_{32} (\log N)^{-\gz},
\label{estimmax}
\end{equation}
and
\begin{multline}
 \label{estimmax2}
\bP\left( M_{N}<  N^{1/\gz} (\log\log N)^{-1} \right) \leq \left( 1- C_{33} N^{-1} (\log\log N)^{\gz} \right)^N \\
    = \exp\left( -(1+o(1)) C_{33}(\log\log N)^{\gz}  \right)
\end{multline} 
Since $\gz>1$, one has from \eqref{estimmax} that the sequence
$\bP\left( M_{2^k}>  2^{(k-1)/\gz} \log 2^{k-1} \right)$ for $k\geq 1$ is summable,
and from \eqref{estimmax2} that the sequence
$\bP\left( M_{2^k}<  2^{(k+1)/\gz} (\log \log 2^{k+1})^{-1} \right)$ is also summable.

\smallskip
The Borel-Cantelli Lemma gives that there exists a random integer $k_0$
such that for all $k\ge k_0$
\begin{equation}
 2^{(k+1)/\gz} (\log \log 2^{k+1})^{-1} \leq M_{2^k}\leq  2^{(k-1)/\gz} \log 2^{k-1} .
\end{equation}
One notices that $(M_N)_{N\geq 0}$ is a non decreasing sequence. Thus, taking $N\geq N_0:=2^{k_0+1}$,
and choosing $k$ such that $2^{k-1}< N \le  2^k$ then one has $k-1\geq k_0$ and so
\begin{equation}
M_N\leq M_{2^k} \le 2^{(k-1)/\gz} \log 2^{k-1}\le N^{1/\gz} \log N,
\end{equation}
and
\begin{equation}
 M_N\geq M_{2^{k-1}} \geq 2^{k/\gz} (\log \log 2^k)^{-1} \geq N^{1/\gz} \log \log N.
\end{equation}
\end{proof}

\end{appendix}

\end{document}